\documentclass{amsart}
\usepackage{amsmath,amssymb,mathrsfs,amsthm}
\usepackage{array,enumerate}
\usepackage{latexsym}
\usepackage{mathrsfs}
\usepackage{amscd}
\usepackage[all]{xy}
\usepackage{graphicx}
\usepackage[svgnames]{xcolor}
\theoremstyle{definition}
\newtheorem{thm}{Theorem}[section]
\newtheorem{dfn}[thm]{Definition}
\newtheorem{not-dfn}[thm]{Notation-Definition}
\newtheorem{nota}[thm]{Notation}
\newtheorem{not-set}[thm]{Notation-Setting}

\newtheorem{prop}[thm]{Proposition}
\newtheorem{cor}[thm]{Corollary}
\newtheorem{lem}[thm]{Lemma}
\newtheorem{rem}[thm]{Remark}

\newtheorem{prop-dfn}[thm]{Proposition-Definition}

\newcommand{\N}{\mathbb{N}}
\newcommand{\Z}{\mathbb{Z}}

\newcommand{\F}{\mathbb{F}}
\newcommand{\G}{\mathbb{G}}
\newcommand{\cC}{\mathcal{C}}

\newcommand{\cO}{\mathcal{O}}

\newcommand{\fm}{\mathfrak{m}}
\newcommand{\fp}{\mathfrak{p}}

\newcommand{\Ab}{\mathrm{Ab}}
\newcommand{\Alg}{\mathrm{Alg}}

\newcommand{\Der}{\mathrm{Der}}
\newcommand{\Coker}{\mathrm{Coker}\,}
\newcommand{\Ima}{\mathrm{Im}\,}
\newcommand{\Ker}{\mathrm{Ker}\,}
\newcommand{\Hom}{\mathop{\mathrm{Hom}}\nolimits}
\newcommand{\Ext}{\mathop{\mathrm{Ext}}\nolimits}
\newcommand{\length}{\mathop{\mathrm{length}}\nolimits}
\newcommand{\Mor}{\mathrm{Mor}}
\newcommand{\Frob}{\mathrm{Frob}}
\newcommand{\Pic}{\mathrm{Pic}}

\newcommand{\id}{\mathrm{id}}
\newcommand{\red}{\mathrm{red}}

\newcommand{\Spec}{\mathrm{Spec}\,}
\newcommand{\sep}{\mathrm{sep}}

\newcommand{\Frac}{\mathrm{Frac}}
\newcommand{\jac}{\mathop{\mathrm{jac}}}
\title[Conductors, Picard schemes, and Jacobian numbers]{On behavior of conductors, Picard schemes, and Jacobian numbers of varieties over imperfect fields}
\author{Ippei Nagamachi,\ Teppei Takamatsu}
\email{nagachi@kurims.kyoto-u.ac.jp,\ teppei@ms.u-tokyo.ac.jp}
\date{}

\begin{document}

\maketitle

\setcounter{section}{-1}

\begin{abstract}
Let $X$ be a regular geometrically integral variety over an imperfect field $K$.
Unlike the case of characteristic $0$, $X':=X\times_{\mathrm{Spec}\,K}\mathrm{Spec}\,K'$ may have singular points for a (necessarily inseparable) field extension $K'/K$.
In this paper, we define new invariants of the local rings of codimension $1$ points of $X'$, and use these invariants for the calculation of $\delta$-invariants (, which relate to genus changes,) and conductors of such points.
As a corollary, we give refinements of Tate's genus change theorem and \cite[Theorem 1.2]{PW}.
Moreover, when $X$ is a curve, we show that the Jacobian number of $X$ is $2p/(p-1)$ times of the genus change by using the above calculation.
In this case, we also relate the structure of the Picard scheme of $X$ with invariants of singular points of $X$.
To prove such a relation, we give a characterization of the geometrical normality of algebras over fields of positive characteristic.
\end{abstract}

\setcounter{tocdepth}{1}
\tableofcontents

\section{Introduction}
Let $K$ be a field, $p\,(\geq0)$ the characteristic of $K$, and $C$ a regular proper geometrically integral curve over $K$.
If $K$ is perfect (in particular, if $p=0$), $C$ is smooth over $K$.
On the other hand, if $K$ is imperfect, $C':=C\times_{\Spec K}\Spec K'$ may have non-regular points for some purely inseparable field extension $K'/K$.
The $\delta$-invariant (cf.\,Notation-Definition \ref{deltaconddef}.3), which relates to genus changes, and the conductor (cf.\,Notation-Definition \ref{deltaconddef}) of such a singular point $c'$, both of which measure how strong the singularity is, have been well-studied.
In \cite{Ta}, Tate proved
$$g_{C}-g_{\widetilde{C'}}\equiv 0\mod \frac{p-1}{2}$$
for any purely inseparable field extension $K'/K$,
where $\widetilde{C'}$ is the normalization of $C'$ in its function field, and $g_{-}$ denotes the arithmetic genus of the curve.
See \cite{Sc} for a modern and simple proof of this theorem.
As a variant of Tate's genus change theorem, it is shown that the $\delta$-invariant is greater than or equal to $(p-1)/2$ in \cite[Theorem 5.7 and Remark 5.9]{IIL} in the case where $K'$ is an algebraic closure of $K$.
Note that $g_{C}-g_{\widetilde{C'}}$ coincides with the sum of the product of the $\delta$-invariants and the residue extension degrees of all the singular points of $C'$.
Patakfalvi and Waldron proved that the conductor divisors of the singular points are divisible by $(p-1)$ (\cite[Theorem 1.2]{PW}).
(Note that, in \cite{PW}, more general varieties are treated.
See also \cite{Tanaka2019}.)
In this paper, we study these two invariants via two distinct approaches.

First, we consider the Picard schemes of algebraic curves.
In \cite{A}, Achet studied the structure of the Picard scheme of a form of $\mathbb{A}^{1}_{K}$, which reflects the properties of the singularity of the complement of $\mathbb{A}^{1}_{K}$ in its regular compactification.
In this paper, we treat general regular singular points of curves.
Suppose that $p>0$ and fix an algebraic closure $\overline{K}$ of $K$.
For any $i\in\N$, write $C(i)$ for the normalization of $C\times_{\Spec K}\Spec K^{1/p^{i}}$ in its function field, $C_{\overline{K}}$ (resp.\,$C(i)_{\overline{K}}$) 
for the algebraic curve 
$C\times_{\Spec K}\Spec\overline{K}$
(resp.\,$C(i)\times_{\Spec K^{1/p}}\Spec\overline{K}$), $g_{C}$ (resp.\,$g_{C(i)}$) for the arithmetic genus of $C$ (resp.\,$C(i)$), and $\widetilde{C}$ for the normalization of $C_{\overline{K}}$ in its function field.

\begin{thm}[See Theorem \ref{overclosed} for the precise statement]
The algebraic group
$$\Ker(\Pic_{C_{\overline{K}}/\overline{K}}\to \Pic_{C(i)_{\overline{K}}/\overline{K}})$$
coincides with the maximal reduced closed subscheme of the algebraic group
$$\Ker(\Pic_{C_{\overline{K}}/\overline{K}}\to \Pic_{\widetilde{C}/\overline{K}})[p^{i}].$$
Here, for a commutative algebraic group $P$, $P[p^{i}]$ denotes the subalgebraic group $\Ker(p^{i}\times\colon P\to P)$ of $P$.
Moreover, (both of) these algebraic groups also coincide with the largest subgroup of $\Ker(\Pic_{C_{\overline{K}}/\overline{K}}\to \Pic_{\widetilde{C}/\overline{K}})$ that is $p^{i}$-torsion, connected, smooth, and unipotent.
\label{intromain1}
\end{thm}

Since the dimension of the Picard scheme of a curve coincides with its genus, this theorem says that the genus change $g_{C(i)}-g_{C}$ can be written simply in terms of algebraic groups, i.e., the dimension of the $p^{i}$-torsion part of the affine part of the Picard scheme of $C_{\overline{K}}$.
To prove Theorem \ref{intromain1}, we study the Weil restriction of $\mathbb{G}_{m,K}$ in Section \ref{affinesection} and the Picard schemes of proper curves with cusps in Section \ref{Picpresection}.
Moreover, in Section \ref{geomnormalsection}, we prove a geometrical normality criterion of rings over imperfect fields which asserts the following:
For any algebra $R$ over $K$, $R$ is geometrically normal over $K$ if and only if $R\otimes_{K}K^{1/p}$ is normal (see Corollary \ref{geomnor}).
We also study conditions that some algebraic subgroups of the affine part of $\Pic_{C'/K'}$ are split over $K'$ in the case of $[K':K]=p$ (see Theorem \ref{x01cor}).

As a corollary of Theorem \ref{intromain1}, we can obtain an inequality
$$g_{C(i+1)}-g_{C(i+2)}\leq g_{C(i)}-g_{C(i+1)}\quad (i\in\N),$$
by applying a discussion similar to the following one:\\
For any $p$-power torsion abelian group $A$, the natural homomorphism
$$p\times\colon A[p^{i+2}]/A[p^{i+1}]\to A[p^{i+1}]/A[p^{i}]$$
is injective, where $A[p^{j}]$ is the subgroup of $p^{j}$-torsion elements of $A$
(cf.\,Corollary \ref{main1cor} and Remark \ref{gpgenuschange}).
In fact, a stronger inequality
$$p(g_{C(i+1)}-g_{C(i+2)})\leq g_{C(i)}-g_{C(i+1)}$$
holds.
As in Corollary \ref{geqp}, we can show this inequality for a geometrically integral discrete valuation ring over $K$ under some mild conditions, by using the elementary theory of discrete valuation rings.
Moreover, this inequality can be improved by using an invariant $q(x)$, which we will introduce later (see Theorem \ref{mainineq}
and Proposition \ref{Bgeqp}).

As we have just seen, the genus changes decrease step-by-step.
On the other hand, the genus changes $g_{C(i+1)}-g_{C(i)}$ can be written as the sum of the local genus changes, which are the products of the $\delta$-invariants and the degrees of the extension of the residue fields, of all the singular points of the curve $C(i)\times_{\Spec K^{1/p^{i}}}\Spec K^{1/p^{i+1}}$.
In the second part of this paper, we study the step-by-step behavior of local genus changes in more detail.
We define invariants $q(x)$ for the local rings of singular points of such curves, which enable us to calculate the $\delta$-invariants and the conductors of the local rings.
In fact, the invariants $q(x)$ will be defined for very general discrete valuation rings $R$ over $K$ (see the settings before Notation-Definition \ref{deltaconddef}).
In Introduction, we limit ourselves to the case where $R$ is essentially of finite type over $K$.

\begin{thm}[See Theorem \ref{q} for the precise statement]
Let $R$ be a discrete valuation ring geometrically integral and essentially of finite type over $K$, $v$ the normalized valuation of $R$, $x$ an element of $K$ satisfying $x^{1/p}\notin K$, $R(1_{x})$ the normalization of $R\otimes_{K}K(x^{1/p})$ in its field of fractions, $\fm(1_{x})$ the maximal ideal of $R(1_{x})$, $e$ the ramification index of $R(1_{x})$ over $R$, and $q(x)$ the natural number
$$\sup_{r\in R}v(r^{p}-x).$$
Suppose that the residue field of $R$ contains $x^{1/p}$.
Write $\cC$ for the conductor of $R\otimes_{K}K(x^{1/p})$ and $\delta$ for the $\delta$-invariant of $R\otimes_{K}K(x^{1/p})$, i.e., the length of $R(1_{x})/R\otimes_{K}K(x^{1/p})$ as an $R(1_{x})$-module.
Then the inequalities $1\leq q(x)<\infty$ hold, and
$e=p$ if and only if $q(x)$ is not divisible by $p$.
Moreover, the following hold:
\begin{description}
\item[The case of $e=p$]
We have
\begin{equation*}
\cC=\fm(1_{x})^{(p-1)(q(x)-1)},\quad
\delta=\frac{(p-1)(q(x)-1)}{2}.
\end{equation*}
\item[The case of $e=1$]
We have
\begin{equation*}
\cC=\fm(1_{x})^{\frac{(p-1)q(x)}{p}},\quad
\delta=\frac{(p-1)q(x)}{2}.
\end{equation*}
\end{description}
\label{intromain2}
\end{thm}

This theorem gives another proof of Tate's genus change formula.
This theorem also gives a generalization of the main theorem of \cite{PW}.
To prove Theorem \ref{intromain2}, we study the structure of the completion of $R(1_{x})$ by using combinatorial techniques.
For example, in the case of $e=p$, we use the classical combinatorial problem, the so-called Frobenius coin problem.
By applying Theorem \ref{intromain2}, we can calculate more precise step-by-step behavior of genus changes of curves (see Theorem \ref{mainineq} and Proposition \ref{Bgeqp}).
We note that we can obtain the relation between conductor $\cC$ and $\delta$,
$$2\length_{R}(R(1_{x})/\cC)=\delta,$$
by Theorem \ref{intromain2} (cf.\,Remark \ref{conductorvsgenuschange}).
As in Remark \ref{conductorvsgenuschange}, this equality follows from theory of dualizing complexes.

Theorem \ref{intromain2} also has an application to the study of the Jacobian number of a curve, which is another invariant for singularities of the curve.
Jacobian numbers have been studied by many researchers, and they have a significant role in studies of singular curves (see
\cite{Buchweitz1980},
\cite{Esteves2003},
\cite{Greuel2007},
\cite{IIL},
and \cite{Tjurina1969}).
Partitioning Jacobian numbers by considering partial field extensions of $K^{1/p}/K$,
we can link Jacobian numbers to the invariants $q(x)$ which we defined in Theorem \ref{intromain2}.
Then we have the following nontrivial relation between Jacobian numbers and genus changes:

\begin{thm}[see Theorem \ref{genjacob} for the precise statement]
Suppose that $[K:K^{p}] < \infty.$ Let $R$ be a discrete valuation ring geometrically integral and essentially of finite type over $K$.
Then we have 
\[
\frac{g_{10}}{(p-1)/2} = \frac{\jac (R)}{p}.
\]
Here, $\jac (R)$ is the Jacobian number of $R$, and $g_{10}$ is the local genus change $\dim_{K^{1/p}}(R(1)/R \otimes_{K} K^{1/p})$, where $R(1)$ is the normalization of the base change $R \otimes_{K}K^{1/p}$ in its field of fractions.
\end{thm}
We note that this theorem clarifies the relation between the smoothness criterion in terms of genus changes (which follows from Tate's genus change theorem, see Remark \cite[Remark 1.9]{IIL}) and that in terms of Jacobian numbers (see \cite[Proposition 4.4]{IIL}).

The content of each section of this paper is as follows:
In Section \ref{affinesection}, we review elementary properties of unipotent algebraic groups over imperfect fields.
In Section \ref{Picpresection}, we study the structure of the Picard schemes of regular geometrically integral curves over imperfect fields.
In Section \ref{geomnormalsection}, we give some elementary lemmas on integral domains over fields of positive characteristics.
In Section \ref{main1section}, we explain structures of $p$-power torsion subgroups of the Picard schemes of regular geometrically integral curves and give the proof of Theorem \ref{intromain2}.
In Section \ref{simplecase}, we give calculations of the conductors and certain local variants of ``genus changes'' of curves which are called $\delta$-invariants.
In Section \ref{eg}, we give examples of discrete valuation rings to show that any behavior of ramification indices appearing in Theorem \ref{mainineq} actually occurs.
In Section \ref{jac}, we define Jacobian numbers and relate them to differential modules.
Moreover, we study the behavior of differential modules via step-by-step field extensions to show Theorem \ref{genjacob}.
\\

\noindent\textbf{Acknowledgements.}
The authors are deeply grateful to Naoki Imai, the supervisor of the second author, for deep encouragement and helpful advice.
They would like to thank Teruhisa Koshikawa for informing them of the arguments used in Lemma \ref{splitvector}.
They are also grateful to Hiromu Tanaka, Tetsushi Ito, Kazuhiro Ito, and Makoto Enokizono for their helpful comments and suggestions.
Moreover, they would like to thank Yuya Matsumoto for helpful comments (especially on the proof of Lemma \ref{coin}).

The first author is supported by JSPS KAKENHI Grant Number 21J00489.
He is also supported by the Iwanami-Fujukai Foundation.
He is also supported by Foundation of Research Fellows, The Mathematical Society of Japan.
The second author is supported by JSPS KAKENHI Grant Number 19J22795.

\subsection*{Notation}
In this paper, all rings are commutative.
Let $p$ be a prime number.
For an algebra $A$ over $\F_{p}$, we write $\Frob_{A}$ for the Frobenius map $A\to A: a\mapsto a^{p}$.
For a ring $R$, we write $\Frac(R)$ for the total ring of fractions of $R$.
For a field $K$ and a scheme $X$ over $K$, we shall say that $X$ is a curve (over $K$) if $X$ is an integral scheme of dimension $1$.
We shall write $\Pic_{X/K}$ (resp.\,$\Pic_{X/K}^{0}$) for the relative Picard scheme of $X$ over $K$ (resp.\,the identity component of $\Pic_{C/K}$).

\section{Notes on algebraic groups in positive characteristics}
In this section, we review elementary properties of unipotent algebraic groups over imperfect fields.

Let $M$ be a field and $H$ an algebraic group over $M$.
Recall that $H$ is called a vector group if $H$ is isomorphic to the product of finite copies of $\mathbb{G}_{a,M}$ over $M$.
We shall say that a smooth connected solvable algebraic group $H$ over $M$ is $M$-split if $H$ admits a composition series
\begin{equation}
\begin{split}
H=H_{0}\supset H_{1}\supset\ldots\supset H_{n}=1
\label{splitdef}
\end{split}
\end{equation}
consisting of smooth closed algebraic subgroups of $H$ such that $H_{i+1}$ is normal in $H_{i}$ and the quotient $H_{i}/H_{i+1}$ is $M$-isomorphic to $\G_{a}$ or $\G_{m}$ for all $0\leq i<n$ (cf.\,{\cite[Appendix B]{CGP}} and Remark \ref{whyassumption}). 
(Note that $H_{i}$ is not necessarily normal in $H$.
On the other hand, in the definition of $M$-splitness in {\cite[Examples 12.3.5. (3)]{Sp}}, $H_{i}$ is suppose to be normal in $H$.
However, we can use results of {\cite{Sp}} because we mainly treat commutative algebraic groups.)
If $M$ is perfect, every connected (commutative) smooth unipotent algebraic group is $M$-split by {\cite[Corollary 14.3.10]{Sp}}.

\begin{rem}
In this paper, we follow the definition of $M$-splitness of {\cite[Appendix B]{CGP}}.
Since an extension of connected (resp.\,smooth) algebraic groups is also connected (resp.\,smooth), an algebraic group admitting a sequence $(\ref{splitdef})$ over $M$ is connected (resp.\,smooth).
Hence, we do not need to assume that $H$ is smooth, connected, and even solvable.
\label{whyassumption} 
\end{rem}

\begin{lem}
Suppose that $H$ is connected, smooth, commutative, and $p$-torsion.
Then $H$ is an $M$-split algebraic group if and only if $H$ is a vector group.
\label{splitvector}
\end{lem}

\begin{proof}
Any vector group is a split algebraic group.
Suppose that $H$ is an $M$-split algebraic group.
By {\cite[Appendix B, Lemma 2.5]{CGP}} and the assumption that $H$ is $p$-torsion, there exists a canonical decomposition $H\simeq V\times W$, where $V$ is a vector subgroup of $H$ and $W$ is a wounded subgroup of $H$, i.e., every homomorphism from $\mathbb{G}_{a,M}$ to $W$ is trivial.
Take a composition series $H=H_{0}\supset H_{1}\supset \ldots\supset H_{n}=1$ of $H$ whose successive quotients are isomorphic to $\mathbb{G}_{a,M}$.
Inductively, it follows that $H_{i}$ is contained in $V$ for each $0\leq i \leq n$.
In particular, we have $H=V$.
\end{proof}

\begin{lem}
Let $i$ be a natural number or $\infty$.
Suppose that $H$ is connected and commutative.
Then the largest subgroup $H_{s}(i)$ of $H$ that is $p^{i}$-torsion, $M$-split, and unipotent exists.
\label{maximalvec}
\end{lem}

\begin{proof}
Let $V_{1}$ and $V_{2}$ be $M$-split $p^{i}$-torsion unipotent algebraic subgroups of $H$.
Then the image of the natural homomorphism $V_{1}\times V_{2}\to H: (v_{1},v_{2})\mapsto v_{1}+v_{2}$ is connected, smooth, and $p^{i}$-torsion.
Since $V_{1}\times V_{2}$ is $M$-split, this image is also $M$-split by {\cite[Exercises 14.3.12. (2)]{Sp}}.
Hence, the desired algebraic subgroup $H_{s}(i)$ exists.
\end{proof}

Next, we discuss the Weil restriction of $\G_{m}$.
Let $K$ be a field, $L/K$ a finite purely inseparable field extension of $K$, and $\overline{L}$ an algebraic closure of $L$.
For a scheme $X$ over $L$, write $(X)_{L/K}$ for the Weil restriction of $X$ to $K$ (if exists).
Recall that, for any scheme $T$ over $K$, we have a natural bijection
$$\Mor_{\Spec L}(T\times_{\Spec K}\Spec L, X)\simeq \Mor_{\Spec K}(T, (X)_{L/K})$$
which is functorial in $T$ and $X$.
Write $(\G_{m})_{L/K}$ for the Weil restriction of $\G_{m,L}$ to $K$ (see {\cite[Section 7.6]{BLR}}).
Then we have a natural homomorphism
$$\eta_{L/K}\colon\G_{m,K}\to (\G_{m})_{L/K}$$
over $K$ which is called the unit of $\G_{m,K}$ and a natural homomorphism
$$\varepsilon_{L/K}\colon(\G_{m})_{L/K}\times_{\Spec K}\Spec L\to\G_{m,L}$$
over $L$ which is called the counit of $\G_{m,L}$.
Note that the composite homomorphism $\varepsilon_{L/K}\circ(\eta_{L/K},\id_{\Spec L})$ coincides with $\id_{\G_{m,L}}$.
Let $L'$ be a field satisfying $K\subset L'\subset L$.
We have a canonical isomorphism $((\G_{m})_{L/L'})_{L'/K}\simeq (\G_{m})_{L/K}$.
Under this identification, it holds that
$$\eta_{L/K}=(\eta_{L/L'})_{L'/K}\circ\eta_{L'/K}.$$

\begin{lem}
The following hold:
\begin{enumerate}
\item
$(\G_{m})_{L/K}$ is connected and smooth.
\item
The sequence
$$0\to\G_{m,K}\overset{\eta_{L/K}}{\to}(\G_{m})_{L/K}\to\Coker\eta_{L/K}\to0$$
gives the extension explained in {\cite[Expos\'e XVII, Th\'eor\`eme 7.2.1]{SGA3}} (or {\cite[Theorem 9.2.2]{BLR}}).
In particular, $\Coker\eta_{L/K}$ is a connected smooth unipotent algebraic group over $K$.
\item
$(\G_{m})_{L/K}\times_{\Spec K}\Spec L$ is a split algebraic group.
\end{enumerate}
\label{forsplit}
\end{lem}

\begin{proof}
Assertion 1 follows from {\cite[Expos\'e XVII Proposition C.5.1]{SGA3}}.
By {\cite[Expos\'e XVII Proposition C.5.1]{SGA3}}, $\varepsilon_{L/K}$ is an epimorphism and $\Ker \varepsilon_{L/K}$ is unipotent.
Since we have $\varepsilon_{L/K}\circ(\eta_{L/K},\id_{\Spec L})=\id_{\G_{m,L}}$, the composite homomorphism
$$\Ker\varepsilon_{L/K} \hookrightarrow (\G_{m})_{L/K}\times_{\Spec K}\Spec L\to (\Coker\eta_{L/K})\times_{\Spec K}\Spec L$$
is an isomorphism.
Thus, $(\Coker\eta_{L/K})$ is a connected smooth unipotent algebraic group over $K$, and assertion 2 holds.

Next, we show assertion 3.
Note that $L\otimes_{K}L$ is an Artin local ring with the $L$-algebra structure given by the base change of $K \rightarrow L$.
Write $\fm$ for the unique maximal ideal of $L\otimes_{K}L$ and, for any natural number $m$, $Q_{m}$ for the functor
$$(-\otimes_{L}(L\otimes_{K}L/\fm^{m}))^{\ast}\colon\Alg_{L}\to \Ab,$$
where $\Alg_{L}$ is the category of $L$-algebras, and $\Ab$ is the category of abelian groups.
Then, $Q_{m}$ is a quotient functor of $Q_{m+1}$ and $\Ker (Q_{m+1} \rightarrow Q_{m})$ is isomorphic to $(-)^{\ast}$ (resp.\,$-\otimes_{L}(\fm^{m}/\fm^{m+1})$) if $m=0$ (resp.\,$m\geq 1$).
Hence, assertion 3 holds.
\end{proof}

\begin{lem}[cf.\,{\cite[Expos\'e XVII Proposition C.5.1]{SGA3}}]
For any field extension $M/K$, the following are equivalent:
\begin{enumerate}
\item
There exists a(n automatically unique injective) $K$-algebra homomorphism $L\hookrightarrow M$.
\item
The base change of $(\G_{m})_{L/K}$ to $M$ becomes a product of a subtorus and a unipotent algebraic subgroup.
(cf.\,Remark \ref{radical})
\item
The base change of $\Coker \eta_{L/K}$ (and hence $(\G_{m})_{L/K}$) to $M$ becomes an $M$-split algebraic group.
\end{enumerate}
\label{Weilstr}
\end{lem}

\begin{proof}
The implication 3$\Rightarrow$2 follows from \cite[Expos\'e XVII Th\'eor\`eme 6.1.1.]{SGA3}, and the implication 1$\Rightarrow$3 follows from Lemma \ref{forsplit}.3.
Therefore, it suffices to show that 1$\Leftrightarrow$2.
As explained in the discussion before Lemma \ref{Weilstr}, $\varepsilon_{L/K}\circ(\eta_{L/K},\id_{\Spec L})$ coincides with $\id_{\G_{m,L}}$.
To show the equivalence, it suffices to show that $L$ is the minimum field that the unipotent radical of $(\G_{m})_{L/K}$ is defined over (cf.\,{\cite[Corollaire (4.8.11)]{EGA42}}).
Let $L'$ be a field satisfying $K\subset L'\subsetneq L$.
By the implication 1$\Rightarrow$2, it suffices to show that the unipotent radical of $(\G_{m})_{L/K}$ is not defined over $L'$.
Since we have $((\G_{m})_{L/L'})_{L'/K}\simeq (\G_{m})_{L/K}$, the natural morphism
$$(\G_{m})_{L/K}\times_{\Spec K}\Spec L'\to (\G_{m})_{L/L'}$$
is an epimorphism by {\cite[Expos\'e XVII Proposition C.5.1]{SGA3}}.
Since the unipotent radical of $(\G_{m})_{L/L'}\times_{\Spec L'}\Spec \overline{L}$ is not defined over $L'$ by {\cite[Expos\'e XVII Corollaire (a) to Proposition C.5.1]{SGA3}}, $(\G_{m})_{L/K}\times_{\Spec K}\Spec L'$ cannot be a product of a subtorus and a unipotent algebraic subgroup.
Hence, we obtain the desired equivalence.
\end{proof}

\begin{rem}
Note that condition 2 in Lemma \ref{Weilstr} is equivalent to the condition that the unipotent radical of $(\G_{m})_{L/K}\times_{\Spec K}\Spec \overline{M}$ is defined over $M$, where $\overline{M}$ is an algebraic closure of $M$.
In the definition of the unipotent radical of an algebraic group $G$ given in the discussion after {\cite[Expos\'e XV D\'efinition 6.1.\,ter.]{SGA3}}, the coefficient field of $G$ is assumed to be an algebraically closed field $k$.
If one defines the unipotent radical of $G$ to be the maximal connected unipotent normal algebraic subgroup of $G$ (in the case where $k$ is not necessarily algebraically closed), condition 2 in Lemma \ref{Weilstr} is equivalent to the condition that $(\G_{m})_{L/K}\times_{\Spec K}\Spec M$ is canonically isomorphic to the direct product of its unipotent radical and its maximal subtorus.
\label{radical}
\end{rem}

\label{affinesection}

\section{The Picard schemes of proper regular geometrically integral  curves}
In this section, we study the structure of the Picard schemes of regular proper geometrically integral curves over imperfect fields.

Let $K$ (resp.\,$K^{\sep}$; $\overline{K}$) be a field (resp.\,a separable closure of $K$; an algebraic closure of $K$).

\begin{lem}
Let $C_{1}$ and $C_{2}$ be proper geometrically integral curves over $K$ and $\alpha\colon C_{1}\to C_{2}$ a universally homeomorphic birational morphism.
Then the natural homomorphism $\Pic_{C_{2}/K}\to\Pic_{C_{1}/K}$ is an epimorphism.
\label{surj}
\end{lem}

\begin{proof}
We may assume that $K=\overline{K}$.
Since $\Pic_{C_{1}/K}$ is smooth over $K$ by {\cite[Proposition 8.4.2]{BLR}}, it suffices to show that the natural homomorphism $\Pic_{C_{2}/K}(K)\to\Pic_{C_{1}/K}(K)$ is surjective (cf.\,\cite[1.71]{Milne2017}).
This homomorphism can be identified with the first homomorphism of the following exact sequence:
$$H^{1}(C_{2},\cO_{C_{2}}^{\ast})\to H^{1}(C_{1},\cO_{C_{1}}^{\ast})\to H^{1}(C_{1},(\alpha_{\ast}\cO_{C_{1}}^{\ast})/\cO_{C_{2}}^{\ast}).$$
Since the support of the sheaf $(\alpha_{\ast}\cO_{C_{1}}^{\ast})/\cO_{C_{2}}^{\ast}$ is $0$-dimensional, we have 
$$H^{1}(C_{1},(\alpha_{\ast}\cO_{C_{1}}^{\ast})/\cO_{C_{2}}^{\ast})=0.$$
Hence, we finish the proof of Lemma \ref{surj}.
\end{proof}

Let $C'$ be a proper geometrically integral curve over $K$ and $C$ the normalization of $C'$ in the function field of $C'$.
Suppose that the natural morphism $\pi\colon C\to C'$ is universally homeomorphic.

\begin{lem}
Suppose that $K=K^{\sep}$.
Let $S\subset C$ be a finite subset of closed points and $n\colon S\to \N$ a map.
For any subset $S'\subset S$, write $C_{S',n}$ for the curve whose underlying topological space is the same as $C$ and whose local ring at $c\in C\setminus S'$ (resp.\,$s\in S'$) is $\cO_{C,c}$ (resp.\,$K+\fm_{s}^{n(s)}$).
Here, $\fm_{s}$ is the maximal ideal of $O_{C,s}$.

\begin{enumerate}
\item
$C_{S,n}$ is a proper geometrically integral curve over $K$.
\item
We have a canonical isomorphism
$$\Ker(\Pic_{C_{S,n}/K}\to\Pic_{C/K} )\simeq\bigoplus_{s\in S} \Ker(\Pic_{C_{\{s\},n}/K}\to\Pic_{C/K}).$$
\end{enumerate}
In the following, suppose that $S=\{s_{0}\}$.
Write $C_{n(s_{0})}$ for $C_{S,n}$, $\fm$ for $\fm_{s_{0}}$, and $L$ for the residue field of $s_{0}$.
Note that we have the following sequence of curves over $K$:
$$C=C_{0}\overset{\pi_{0}}{\to} C_{1}\to\ldots\to C_{i}\overset{\pi_{i}}{\to}\ldots.$$
\begin{enumerate}
\setcounter{enumi}{2}
\item
For any positive integer $i$, we have a canonical isomorphism between an fpqc sheaf defined by $\fm^{i}/\fm^{i+1}$ and that of $\Ker(\Pic_{C_{i+1}/K}\to\Pic_{C_{i}/K}).$
In particular, $\Ker(\Pic_{C_{i+1}/K}\to\Pic_{C_{i}/K})$ is isomorphic to $\G_{a}^{[L:K]}$ over $K$.
\item
We have a canonical isomorphism
$$\Coker(\eta\colon \G_{m,K}\to (\G_{m})_{L/K}) \simeq \Ker(\Pic_{C_{1}/K}\to\Pic_{C/K}).$$
\end{enumerate}
\label{connsmunip}
\end{lem}

\begin{proof}
For assertion 1, we only show that $C_{S,n}$ is a scheme of finite type over $K$.
Take an affine open subscheme $U=\Spec B$ of $C$ such that $U\cap S=\{s\}$.
It suffices to show that $(K+\fm_{s}^{n(s)})\cap B$ is a $K$-algebra of finite type.
Since $B/(B\cap\fm_{s}^{n(s)})$ is a finite dimensional $K$-linear space, $B$ is finite over $(K+\fm_{s}^{n(s)})\cap B$.
Then, since $B$ is of finite type over $K$, $(K+\fm_{s}^{n(s)})\cap B$ is of finite type over $K$.

Next, we show assertion 3.
Write $f_{i+1}$ for the structure morphism $C_{i+1}\to \Spec K$.
From the exact sequence
$$0\to\cO_{C_{i+1}}^{\ast} \to\pi_{i \ast}\cO_{C_{i}}^{\ast}\to(\pi_{i\ast}\cO_{C_{i}}^{\ast})/\cO_{C_{i+1}}^{\ast} \to 0,$$
we obtain an exact sequence of fpqc sheaves over $\Spec K$
$$0\to f_{i+1\ast}((\pi_{i\ast}\cO_{C_{i}}^{\ast})/\cO_{C_{i+1}}^{\ast}) \to R^{1}(f_{i+1})_{\ast}\cO_{C_{i+1}}^{\ast} \to R^{1}f_{i\ast}\cO_{C_{i}}^{\ast}  \to 0.$$
This exact sequence of fpqc sheaves can be identified with the sequence of fpqc sheaves defined by the following exact sequence of algebraic groups over $K$:
$$0\to \Ker(\Pic_{C_{i+1}/K}\to\Pic_{C_{i}/K})\to\Pic_{C_{i+1}/K}\to\Pic_{C_{i}/K}\to 0.$$
Hence it suffices to show that the fpqc sheaf $f_{i+1, \ast}((\pi_{i\ast}\cO_{C_{i}}^{\ast})/\cO_{C_{i+1}}^{\ast})$ on $\Spec K$ is canonically isomorphic to that defined by $\fm^{i}/\fm^{i+1}$.
For any $K$-algebra $A$, we have a canonical group isomorphism
$$(\fm^{i}/\fm^{i+1})\otimes_{K}A\simeq ((K+\fm^{i})\otimes_{K}A)^{\ast}/((K+\fm^{i+1})\otimes_{K}A)^{\ast}$$
which sends $m$ to $1+m$.
Here, the right-hand group coincides with the group of sections of the quotient presheaf of $\pi_{i\ast}\cO_{C_{i}}^{\ast}$ by $\cO_{C_{i+1}}^{\ast}$.
Therefore, assertion 3 holds.

Assertion 2 follows from a similar discussion to that for assertion 3 and the fact that the quotient sheaf $O_{C}^{\ast}/O_{C_{S,n}}^{\ast}$ is isomorphic to the direct sum of the skyscraper sheaves $O_{C,s}^{\ast}/O_{C_{S,n},s}^{\ast}$ at $s$ for all $s\in S$.

Finally, we show assertion 4.
Write $R$ for the local ring $\cO_{C,s_{0}}$.
For any $K$-algebra $A$, we have a commutative diagram with exact horizontal lines
$$
\xymatrix{
0\ar[r]
&((K+\fm)\otimes_{K} A)^{\ast}\ar[r]\ar[d]
&(R\otimes_{K} A)^{\ast}\ar[r]\ar[d]
&(R\otimes_{K} A)^{\ast}/((K+\fm)\otimes_{K} A)^{\ast}\ar[r]\ar[d]
&0\\
0\ar[r]
&A^{\ast}\ar[r]
&(L\otimes_{K}A)^{\ast}\ar[r]
&(L\otimes_{K}A)^{\ast}/A^{\ast}\ar[r]
&0.
}
$$
Here, the both first and second vertical arrows are surjective and the kernels of these homomorphisms are isomorphic to $1+\fm\otimes_{K}A$.
Therefore, the third vertical arrow is an isomorphism.
Then assertion 4 follows from an argument similar to that of the proof of assertion 3.
\end{proof}

\begin{prop}[cf.\,{\cite[Proposition 9.2.9]{BLR}}]
$\Ker(\Pic_{C'/K}\to\Pic_{C/K})$ is a connected smooth unipotent algebraic group over $K$.
\label{connunip}
\end{prop}

\begin{rem}
{\cite[Proposition 9.2.9]{BLR}} treats the case where $K$ is perfect and states that $\Ker(\Pic_{C'/K}\to\Pic_{C/K})$ is unipotent.
In the proof of {\cite[Proposition 9.2.9]{BLR}}, results of {\cite{Ser1}} are used and we can see that $\Ker(\Pic_{C'/K}\to\Pic_{C/K})$ is also connected and smooth.
\end{rem}

\begin{proof}[Proof of Proposition \ref{connunip}]
We may and do assume that $K=K^{\sep}$.
Write $S$ for the closed subset of $C$ where $\pi$ is not locally isomorphic.
For every $s\in S$, fix a natural number $n(s)$ such that $\fm_{s}^{n(s)}\subset \cO_{C',\pi(s)}$.
Then we have a natural morphism $C'\to C_{S,n}$ over $K$.
The induced homomorphism
$$\Ker(\Pic_{C_{S,n}/K}\to\Pic_{C/K})\to\Ker(\Pic_{C'/K}\to\Pic_{C/K})$$
is an epimorphism by Lemma \ref{surj}.
Hence, we may and do assume that $C'=C_{S,n}$.
In this case, $\Ker(\Pic_{C_{S,n}/K}\to\Pic_{C/K} )$ is a successive extension of connected smooth unipotent algebraic groups by Lemmas \ref{connsmunip}.2, \ref{connsmunip}.3, and \ref{connsmunip}.4 and Lemma \ref{Weilstr}.
Therefore, $\Ker(\Pic_{C_{S,n}/K}\to\Pic_{C/K})$ is a connected smooth unipotent algebraic group over $K$.
\end{proof}

\label{Picpresection}

\section{A geometrical normality criterion}
In this section, we introduce some notions on integral domains over fields of positive characteristics, which we use in Section \ref{simplecase}.
We also give an elementary lemma (Lemma \ref{element}) which we use in the proof of Theorem \ref{overclosed}.
To prove Theorem \ref{overclosed}, we only need Lemma \ref{element} in this section.
However, we discuss general ring theory and prove a geometrical normality criterion (cf.\,Corollary \ref{geomnor}) because the authors cannot find it in literature.

We start with reviews of some elementary facts on rings.
Let $k$ be a field and $A$ a $k$-algebra.
Recall that we say that $A$ is normal if, for any prime ideal $\fp$ of $A$, $A_{\fp}$ is an integral domain which is integrally closed in $\Frac(A_{\fp})$.
If $A$ is normal, then $A$ is reduced and integrally closed in $\Frac(A)$.
Note that the converse is true if $A$ has finitely many minimal primes.
Recall that we say that $A$ is geometrically reduced (resp.\,geometrically normal) over $k$ if, for any field extension $k'\supset k$, $A\otimes_{k}k'$ is reduced (resp.\,normal).
In the case where the characteristic of $k$ is $p>0$, $A$ is geometrically reduced (resp.\,geometrically normal) over $k$ if and only if $A\otimes_{k}k^{1/p^{\infty}}$ is reduced (resp.\,normal) by {\cite[Tag 030V and 037Z]{Stacks}} (cf.\,see also {\cite[\S4.6]{EGA42}}).

Suppose that the characteristic of $k$ is $p$.
Note that $A$ is reduced if and only if $\Frob_{A}$ is injective.
In this case, we write $A^{1/p}$ for the target $A$ of $\Frob_{A}$ if we want to identify $A$ with the image of $\Frob_{A}$.
Suppose that $A$ is geometrically reduced.
\begin{lem}
For any natural number $m$, we have natural inclusions
$$A\hookrightarrow A\otimes_{k}k^{1/p^{m}}\hookrightarrow A^{1/p^{m}}.$$
\label{injections}
\end{lem}

\begin{proof}
Since $A\otimes_{k}k^{1/p^{m}}$ is reduced, the composite homomorphism
$$
A\otimes_{k}k^{1/p^{m}}\to A^{1/p^{m}}
\to A^{1/p^{m}}\otimes_{k^{1/p^{m}}}k^{1/p^{2m}}
=(A\otimes_{k}k^{1/p^{m}})^{1/p^{m}}
$$
is injective.
Hence, the desired injectivity for the second homomorphism holds.
\end{proof}
Note that the inclusions in Lemma \ref{injections} induce inclusions
$$\Frac(A)\hookrightarrow \Frac(A\otimes_{k}k^{1/p^{m}})\hookrightarrow \Frac(A^{1/p^{m}}).$$
This follows from the fact that an element $a\in A^{1/p^{m}}$ is regular (in $A^{1/p^{m}}$) if and only if $a^{p^{m}}\in A$ is regular (in $A$).
Moreover, it holds that
$$\Frac(A)\otimes_{k}k^{1/p^{m}}\simeq \Frac(A\otimes_{k}k^{1/p^{m}}).$$
In particular, we have
\begin{equation}
\begin{split}
(\Frac(A))\otimes_{k}k^{1/p^{\infty}}\simeq \Frac(A\otimes_{k}k^{1/p^{\infty}}).
\end{split}
\label{geomnormalfld}
\end{equation}

\begin{lem}
Let $K$ be a field of characteristic $p$, $R$ a normal $K$-algebra, and $m$ a natural number.
Suppose that $R$ is geometrically reduced over $K$.
Write $R(m)$ for the normalization of $R\otimes_{K}K^{1/p^{m}}$ in $\Frac(R\otimes_{K}K^{1/p^{m}})$.
Let $L$ be a field containing $K^{1/p^{m}}$.
Consider the following diagram:
$$
\xymatrix{
L\ar@{^{(}->}[r]
&R\otimes_{K}L\ar@{^{(}->}[r]
&R(m)\otimes_{K^{1/p^{m}}}L\ar@{^{(}->}[r]
&R^{1/p^{m}}\otimes_{K^{1/p^{m}}}L\\
K^{1/p^{m}}\ar@{^{(}->}[r]\ar@{^{(}->}[u]
&R\otimes_{K}K^{1/p^{m}}\ar@{^{(}->}[r]\ar@{^{(}->}[u]
&R(m)\ar@{^{(}->}[u]\ar@{^{(}->}[r]
&R^{1/p^{m}}\ar@{^{(}->}[u]\\
K\ar@{^{(}->}[r]\ar@{^{(}->}[u]
&R.\ar@{^{(}->}[u]
&
&
}
$$
Then we have
\begin{equation*}
\begin{split}
R(m)\otimes_{K^{1/p^{m}}}L&=\{x\in\Frac(R\otimes_{K}L)\mid x^{p^{m}}\in R\otimes_{K}L\}\\
&=(\Frac(R\otimes_{K}L))\cap(R\otimes_{K}L)^{1/p^{m}}.
\end{split}
\end{equation*}
\label{element}
\end{lem}

\begin{proof}
First, we show the assertion in the case where $L=K^{1/p^{m}}$.
Since $R^{1/p^{m}}$ coincides with the integral closure of $R$ in $\Frac(R^{1/p^{m}})$, it holds that
\begin{equation}
R(m)= \Frac(R\otimes_{K}K^{1/p^{m}})\cap R^{1/p^{m}} (\subset\Frac(R^{1/p^{m}})).
\label{pmcase}
\end{equation}

Next, we consider the case where $L$ is perfect.
We show that we have
\begin{equation*}
\begin{split}
R(m)\otimes_{K^{1/p^{m}}}L&=(\Frac(R\otimes_{K}K^{1/p^{m}})\cap R^{1/p^{m}})\otimes_{K^{1/p^{m}}}L\\
&=((\Frac(R\otimes_{K}K^{1/p^{m}}))\otimes_{K^{1/p^{m}}}L)\cap(R^{1/p^{m}}\otimes_{K^{1/p^{m}}}L)\\
&=(\Frac(R\otimes_{K}L))\cap(R^{1/p^{m}}\otimes_{K^{1/p^{m}}}L)\\
&=(\Frac(R\otimes_{K}L))\cap(R\otimes_{K}L)^{1/p^{m}}.
\end{split}
\end{equation*}
The first equality follows from the formula (\ref{pmcase}).
The second equality holds since $L$ is flat over $K^{1/p^{m}}$.
The third equality follows from the isomorphism (\ref{geomnormalfld}).
Since $L$ is perfect, we have $(R\otimes_{K}L)^{1/p^{m}}=R^{1/p^{m}}\otimes_{K^{1/p^{m}}}L$.
Thus, the fourth equality holds.

Finally, we consider a general $L$.
Let $x$ be an element of $\Frac(R\otimes_{K}L)$.
Since $L^{1/p^{\infty}}$ is faithfully flat over $L$, $x^{p^{m}}\in R\otimes_{K}L$ (resp.\,$x\in R(m)\otimes_{K^{1/p^{m}}}L$) if and only if $x^{p^{m}}\in R\otimes_{K}L^{1/p^{\infty}}$ (resp.\,$x\in R(m)\otimes_{K^{1/p^{m}}}L^{1/p^{\infty}}$).
Hence, we finish the proof of Lemma \ref{element}.
\end{proof}

\begin{cor}
Let $K$ be a field of characteristic $p$ and $R$ an algebra over $K$.
Then $R$ is geometrically normal if and only if $R\otimes_{K}K^{1/p}$ is normal.
\label{geomnor}
\end{cor}

\begin{proof}
If $R$ is geometrically normal, $R\otimes_{K}K^{1/p}$ is normal by the definition of the geometrical normality.
Suppose that $R\otimes_{K}K^{1/p}$ is normal.
It suffices to show that $R\otimes_{K}K^{1/p^{\infty}}$ is normal.
Thus, we may and do assume that $R$ is a local ring and $R\otimes_{K}K^{1/p}$ is a normal local domain.
To show that the integral domain $R\otimes_{K}K^{1/p^{\infty}}$ is normal, it suffices to show that
$$\{r\in\Frac(R\otimes_{K}K^{1/p^{\infty}})\mid \exists i\in\Z, r^{p^{i}}\in R\otimes_{K}K^{1/p^{\infty}}\}=R\otimes_{K}K^{1/p^{\infty}}.$$
Hence, Corollary \ref{geomnor} follows from Lemma \ref{element} and the assumption that $R\otimes_{K}K^{1/p}$ is normal.
\end{proof}

\begin{not-dfn}[cf.\,Remark \ref{B?}]
Let $K$ be a field of characteristic $p$, $B$ a subset of a $p$-basis of $K$ over $K^{p}$, and $R$ an algebra over $K$.
We shall write $K(B^{1/p^{n}})$ (resp.\,$K(B^{1/p^{\infty}})$) for the field $K(x^{1/p^{n}}\mid x\in B)$ (resp.\,$\bigcup_{n\geq 1}K(B^{1/p^{n}})$).
We shall say that $R$ is {\it $B$-reduced} (resp.\,{\it $B$-integral}; {\it $B$-normal}) if $R\otimes_{K}K(B^{1/p^{\infty}})$ is reduced (resp.\,integral; normal).
For a $B$-integral algebra $R$, we shall say that $R$ is {\it $B$-Japanese} if, for any natural number $n$, the normalization of $R\otimes_{K}K(B^{1/p^{n}})$ in its field of fractions is finite over $R\otimes_{K}K(B^{1/p^{n}})$.
\label{x-geom}
\end{not-dfn}

\begin{rem}
We use the same notation in Notation-Definition \ref{x-geom}.
Let $\mathscr{P}$ be one of the properties written in italics in Notation-Definition \ref{x-geom}.
Note that, for any $p$-basis $B'$ of $K^p(B)$ over $K^p$, $R$ is $B$-$\mathscr{P}$ if and only if $R$ is $B'$-$\mathscr{P}$.
However, we do not adopt the terminology ``$K(B^{1/p})$-$\mathscr{P}$" because these could be misunderstood as properties of algebras over $K(B^{1/p})$.
\label{B?}
\end{rem}

\begin{prop}
Let $K$, $B$, and $R$ be as in Notation-Definition \ref{x-geom}.
\begin{enumerate}
\item 
$R$ is $B$-integral if and only if $\Spec R$ is irreducible and $R$ is $B$-reduced.
\item
$R$ is $B$-reduced (resp.\,$B$-normal) if and only if $R\otimes_{K}K(B^{1/p})$ is reduced (resp.\,normal).
\item
$R$ is $B$-reduced (resp.\,$B$-normal) if and only if $R$ is $B'$-reduced (resp.\,$B'$-normal) for any finite subset $B'\subset B$.
\end{enumerate}
Suppose that $R$ is a $B$-integral discrete valuation ring and write $R(n_{B})$ for the normalization of $R\otimes_{K}K(B^{1/p^{n}})$ in its field of fractions.
\begin{enumerate}
\setcounter{enumi}{3}
\item
$R(n_{B})$ is a discrete valuation ring.
Moreover, if $R$ is $B$-Japanese, then $R\otimes_{K}K(B^{1/p^{n}})$ is Noetherian.
\item
If $R$ is $B$-Japanese, then $R$ is $B'$-Japanese for any subset $B'\subset B$.
\item
$R$ is $B$-Japanese if and only if $R(1_{B})$ is finite over $R\otimes_{K}K(B^{1/p})$.
\item
If $R$ is $B$-Japanese, then the completion $\widehat{R}$ is $B$-integral.
\end{enumerate}
\label{B-prop}
\end{prop}

\begin{proof}
Since the morphism $\Spec R\otimes_{K}K(B^{1/p^{\infty}})\to \Spec R$ is a homeomorphism, assertion 1 holds.

Next, we show assertion 2.
If $R$ is $B$-reduced (resp.\,$B$-normal),
$R\otimes_{K}K(B^{1/p})$ is also reduced (resp.\,normal) since $K(B^{1/p^{\infty}})$ is faithfully flat over $K(B^{1/p})$.
Suppose that $R\otimes_{K}K(B^{1/p})$ is reduced, or equivalently, the absolute Frobenius endomorphism $F$ of $R\otimes_{K}K(B^{1/p})$ is injective.
Note that the image of $F$ is contained in $R$.
Then the absolute Frobenius endomorphism $F'$ of $R\otimes_{K}K(B^{1/p^{2}})$ can be identified with $\bigoplus_{\varphi}F$ via the commutative diagram
$$
\xymatrix{
R\otimes_{K}K(B^{1/p^{2}})\ar[r]_-{F'}\ar@{=}[d]
&R\otimes_{K}K(B^{1/p^{2}})\\
\underset{\varphi}\bigoplus ((R\otimes_{K}K(B^{1/p}))\underset{b\in B}{\prod}b^{\varphi(b)/p^{2}})\ar[r]_-{\bigoplus_{\varphi}F}
&\underset{\varphi}\bigoplus (R\underset{b\in B}{\prod}b^{\varphi(b)/p}),\ar@{_{(}->}[u]
}
$$
where $\varphi$ ranges over the maps from $B$ to $\{0,\ldots,p-1\}$ such that $\varphi^{-1}(\{1,\ldots,p-1\})$ is a finite set.
Therefore, $F'$ is injective, and hence $R$ is $B$-reduced by induction.

Next, suppose that $R\otimes_{K}K(B^{1/p})$ is normal.
Note that, by the above argument, $R$ is $B$-reduced.
To show that $R$ is $B$-normal, it suffices to show that the first and second parts of the proof of Lemma \ref{element} work even if we replace $L, R(1), K^{1/p^{m}}$, and the assumption that $R$ is geometrically reduced over $K$ with $K(B^{1/p^{\infty}}), R(1_{B}), K(B^{1/p^{m}})$, and the assumption that $R$ is $B$-reduced, respectively.
The only thing we should confirm is the validity of the fourth equality which we checked in the second part of the proof of Lemma \ref{element}.
Hence, we should show
\begin{align*}
R^{1/p}\otimes_{K(B^{1/p})}K(B^{1/p^{\infty}})
=(R\otimes_{K}K(B^{1/p^{\infty}}))^{1/p}.
\end{align*}
This follows from
\begin{align*}
R^{1/p}\otimes_{K(B^{1/p})}K(B^{1/p^{\infty}})
&=(R\otimes_{K^{p}(B)}K^{p}(B^{1/p^{\infty}}))^{1/p}\\
&=(R\otimes_{K}K\otimes_{K^{p}(B)}K^{p}(B^{1/p^{\infty}}))^{1/p}\\
&=(R\otimes_{K}K(B^{1/p^{\infty}}))^{1/p}.
\end{align*}

Assertion 3 follows from the fact that
$$R\otimes_{K}K(B^{1/p})=\varinjlim_{B'}R\otimes_{K}K(B'^{1/p}),$$
where $B'$ ranges over the finite subsets of $B$.

Next, we prove assertion 4.
Write $v_{R}$ for the normalized valuation of $R$ and $v_{R(n_{B})}$ for the map
$$R(n_{B})\to \N\cup\{\infty\}: r\mapsto v_{R}(r^{p^{n}}).$$
Then $v_{R(n_{B})}$ is a (possibly not normalized) nontrivial discrete valuation of $R(n_{B})$, which shows that $R(n_{B})$ is a discrete valuation ring.
If $R$ is $B$-Japanese, $R\otimes_{K}K(B^{1/p^{n}})$ is Noetherian by the Eakin-Nagata theorem.

Next, we suppose that $R$ is $B$-Japanese and prove assertion 5.
Since $R$ is $B$-integral, we have natural injections
$$(R\otimes_{K}K(B'^{1/p^{n}}))\otimes_{K(B'^{1/p^{n}})}K(B^{1/p^{n}})
\hookrightarrow
R(n_{B'})\otimes_{K(B'^{1/p^{n}})}K(B^{1/p^{n}})
\hookrightarrow
R(n_{B}).$$
Since $K(B^{1/p^{n}})$ is faithfully flat over $K(B'^{1/p^{n}})$ and $R\otimes_{K}K(B^{1/p^{n}})$ is Noetherian by 4, $R(n_{B'})$ is finite over $R\otimes_{K}K(B'^{1/p^{n}})$.
Hence, assertion 5 holds.

Next, we prove assertion 6.
By the proof of assertion 4, $R\otimes_{K}K(B^{1/p})$ is Noetherian.
It suffices to show that $R(2_{B})$ is finite over $R(1_{B})\otimes_{K(B^{1/p})}K(B^{1/p^{2}})$, which is equivalent to that $R(2_{B})^{p}$ is finite over $R(1_{B})^{p}\otimes_{K^{p}(B)}K^{p}(B^{1/p})$.
Since $R$ is faithfully flat over $R(1_{B})^{p}$, it suffices to show that $R(2_{B})^{p}\otimes_{R(1_{B})^{p}}R$ is finite over $$(R(1_{B})^{p}\otimes_{K^{p}(B)}K^{p}(B^{1/p}))\otimes_{R(1_{B})^{p}}R
=R\otimes_{K}K(B^{1/p}).$$
Since $R(2_{B})^{p}\otimes_{R(1_{B})^{p}}R$ is integral over $R\otimes_{K}K(B^{1/p})$ and $R$ is flat over $R(1)^{p}$, $R(2_{B})^{p}\otimes_{R(1_{B})^{p}}R$ is a subring of $R(1_{B})$.
Since $R(1_{B})$ is finite over a Noetherian ring $R\otimes_{K}K(B^{1/p})$, assertion 6 holds.

Finally, we assume that $R$ is a $B$-Japanese discrete valuation ring and prove assertion 7.
We may assume that $B$ is a finite set by 3 and 5.
By 1 and 2, it suffices to show that $\widehat{R}\otimes_{K}K(B^{1/p})=\widehat{R}\otimes_{R}(R\otimes_{K}K(B^{1/p}))$ is reduced.
Since $R\otimes_{K}K(B^{1/p})$ and hence $R(1_{B})$ is finite over $R$, $\widehat{R}\otimes_{R}(R\otimes_{K}K(B^{1/p}))=\widehat{R\otimes_{K}K(B^{1/p})}$ is a subring of $\widehat{R}\otimes_{R}(R(1_{B}))=\widehat{R(1_{B})}$.
Since $\widehat{R(1_{B})}$ is a discrete valuation ring, assertion 7 holds. 
\end{proof}

The next corollary will be improved in Proposition \ref{Bgeqp}.
We note that $g_{10}$ in Corollary \ref{geqp} is a local analogue of a genus change of a curve.

\begin{cor}
Let $K$ be a field of characteristic $p$, $B$ a $p$-basis of $K$ over $K^{p}$, and $R$ a discrete valuation ring geometrically integral over $K$.
Moreover, we use the notation of Proposition \ref{B-prop}.
Suppose that $R$ is $B$-Japanese and the residue field $L$ of $R$ is finite over $K$.
Write $g_{10}$ (resp.\,$g_{21}$) for
\begin{equation*}
\begin{split}
&\dim_{K^{1/p}}R(1)/R\otimes_{K}K^{1/p}\\
(\text{resp.\,}&\dim_{K^{1/p^{2}}}R(2)/R(1)\otimes_{K^{1/p}}K^{1/p^{2}}).
\end{split}
\end{equation*}
Then we have $g_{10}\geq pg_{21}$.
\label{geqp}
\end{cor}

\begin{proof}
By the proof of Proposition \ref{B-prop}.6, it suffices to show that
$$\dim_{K^{1/p}}R\otimes_{R(1)^{p}}R(2)^{p}/R\otimes_{K}K^{1/p} \geq pg_{21}.$$
Since we have
\begin{equation*}
\begin{split}
g_{21}=&\dim_{K^{1/p^{2}}}R(2)/R(1)\otimes_{K^{1/p}}K^{1/p^{2}}\\
=&\dim_{K^{1/p}}R(2)^{p}/R(1)^{p}\otimes_{K}K^{1/p}.
\end{split}
\end{equation*}
and
\begin{equation*}
\begin{split}
R\otimes_{R(1)^{p}}R(2)^{p}/R\otimes_{K}K^{1/p}=&(R(2)^{p}/R(1)^{p}\otimes_{K}K^{1/p})\otimes_{R(1)^{p}}R,
\end{split}
\end{equation*}
Corollary \ref{geqp} follows from the next lemma.
\end{proof}

\begin{lem}
The residue degree or ramification index of the extension $R(1)^{p}\subset R$ is greater than $1$.
\end{lem}

\begin{proof}
Write $L(1)$ for the residue field of $R(1)$.
We may and do assume that $L(1)=L^{1/p}$.
It suffices to show that the ramification index of $R(1)$ over $R$ is $1$.
Let $B'$ be a subset of $K^{1/p}$ which is a $p$-basis of $L K^{1/p} / L$.
Then $R\otimes_{K}K(B')$ is a discrete valuation ring whose residue field is $LK^{1/p}$ and the ramification index of the extension $R \otimes_K K(B') \supset R$ is $1$ by {\cite[I \S6 Proposition 15]{Ser}}.
Let $B''$ be a subset of $K^{1/p}$ such that $B'\cup B''$ is a $p$-basis of $K^{1/p}/K$ and $B'\cap B''=\emptyset$.
Note that the cardinality of $B''$ is $\log_{p}[K^{1/p}:K(B')]$.
Since we have
$$[K^{1/p}:K(B')]\leq[L(B'):K(B')]\leq[L:K],$$
$B''$ is a finite set.
Moreover, it holds that
\begin{equation*}
\begin{split}
&[L^{1/p}:L(B')]=[L^{1/p}(B'^{1/p^{\infty}}):L(B'^{1/p^{\infty}})]\\
=&[K^{1/p}(B'^{1/p^{\infty}}):K(B'^{1/p^{\infty}})]=[K^{1/p}:K(B')],
\end{split}
\end{equation*}
where the second equality holds since we have $L^{1/p}(B'^{1/p^{\infty}})=(L(B'^{1/p^{\infty}}))^{1/p}$ is finite over $K^{1/p}(B'^{1/p^{\infty}})=(K(B'^{1/p^{\infty}}))^{1/p}$.
Since we have $L(1)=L^{1/p}$, the ramification index of $R(1)$ over $R\otimes_{K}K(B')$ is $1$.
\end{proof}

\label{geomnormalsection}

\section{The Picard schemes of regular curves}
In this section, we explain structures of $p$-power torsion subgroups of the Picard schemes of regular geometrically integral curves.

Let $K$ be a field of characteristic $p$, $\overline{K}$ an algebraic closure of $K$, and $C$ a proper regular curve over $K$.
For any natural number $m$, write $C(m)$ for the normalization of $C\times_{\Spec K}\Spec K^{1/p^{m}}$ in its field of fractions and $g_{m}$ for the arithmetic genus of $C(m)$.
Moreover, for any field extension $L$ over $K^{1/p^{m}}$, we write $C(m)_{L}$ for the scheme $C(m)\times_{\Spec K^{1/p^{m}}}\Spec L$.
Note that we have the following sequence of curves over $\overline{K}$:
$$C(m)_{\overline{K}}\to C(m-1)_{\overline{K}}\to \ldots \to C(1)_{\overline{K}}\to C(0)_{\overline{K}}(=C\times_{\Spec K}\Spec \overline{K}).$$
There exists a natural number $n$ such that $C(n)_{\overline{K}}$ is normal.
Write $n_{0}$ for the minimal one.

First, we recall fundamental properties of the Picard schemes of curves.
\begin{lem}
Let $m$ be a natural number.
We have the following:
\begin{enumerate}
\item
$\Pic_{C(m)/K^{1/p^{m}}}$ is a smooth group scheme of dimension $g_{m}$ over $K^{1/p^{m}}$.
\item
The N\'eron-Severi group $\Pic_{C(m)/K^{1/p^{m}}}(\overline{K})/\Pic^{0}_{C(m)/K^{1/p^{m}}}(\overline{K})$ of $C(m)$ is isomorphic to $\Z$.
\item
There exists a short exact sequence of algebraic groups over $\overline{K}$
$$0\to G(m) \to \Pic^{0}_{C(m)_{\overline{K}}/\overline{K}} \to A(m) \to 0$$
such that $A(m)$ is an abelian variety and $G(m)$ is a connected smooth unipotent algebraic group.
Moreover, these properties characterize the algebraic subgroup $G(m)$ of $\Pic^{0}_{C(m)_{\overline{K}}/\overline{K}}$ and the quotient algebraic group $A(m)$ of $\Pic^{0}_{C(m)_{\overline{K}}/\overline{K}}$.
\item
The natural homomorphism $\Pic_{C(m)_{\overline{K}}/\overline{K}}\to \Pic_{C(m+1)_{\overline{K}}/\overline{K}}$ is surjective.
This homomorphism induces an isomorphism $A(m)\to A(m+1)$ and an isomorphism between the N\'eron-Severi groups of $C(m)$ and $C(m+1)$.
\item
$\Pic^{0}_{C(n_{0})/K^{1/p^{n_{0}}}}$ is an abelian variety over $K^{1/p^{n_{0}}}$.
In other words, $G(n_{0})$ is a trivial group scheme and we have $\Pic^{0}_{C(n_{0})_{\overline{K}}/\overline{K}} \simeq A(n_{0})$.
\end{enumerate}
\label{fundamentalstructure}
\end{lem}

\begin{rem}
By assertion 4, $A(m)\simeq A(0)$ for any natural number $m$.
\end{rem}

\begin{proof}[Proof of Lemma \ref{fundamentalstructure}]
Assertion 1 follows from {\cite[$\mathrm{n}^{\circ}$ 232, Theorem 3.1]{FGA}} (or {\cite[Theorem 8.2.1]{BLR}}) and {\cite[Proposition 8.4.2]{BLR}}.
Assertion 2 follows from {\cite[Corollary 9.2.14]{BLR}}.
Assertion 5 follows from {\cite[Proposition 9.2.3]{BLR}}.
Assertions 3 and 4 follow from assertion 1, {\cite[Theorem 9.2.1]{BLR}}, Proposition \ref{connunip}, and assertion 5.
\end{proof}

For any natural number $0 \leq m\leq n_{0}$, let $G(m)$ be as in Lemma \ref{fundamentalstructure}.
Write $A$ (resp.\,$G$) for $A(0)$ (resp.\,$G(0)$) in Lemma \ref{fundamentalstructure}.
By Lemma \ref{fundamentalstructure}, we have a sequence of connected smooth unipotent commutative algebraic groups
$$G=G(0)\to G(1)\to \ldots\to G(n_{0})=0.$$
Write $G_{i}$ for the kernel of the homomorphism $G\to G(i)$.
Then $G$ has an increasing filtration
$$0=G_{0}\hookrightarrow G_{1}\hookrightarrow\ldots\hookrightarrow G_{n_{0}}=G.$$

\begin{thm}[cf.\,Theorem \ref{intromain1}]
Let $0\leq i \leq n_{0}$ be a natural number.
The following four algebraic subgroups $G_{i},G'_{i},G''_{i}, G'''_{i}$ of $G$ coincide:
\begin{itemize}
\item
$G_{i}(=\Ker(G\to G(i)))$.
\item
$G'_{i}:=\Ker(\Pic_{C_{\overline{K}}/\overline{K}}\to\Pic_{C(i)_{\overline{K}}/\overline{K}})$.
\item
$G''_{i}:=G_{s}(i)$, i.e., the largest subgroup of $G$ that is $p^{i}$-torsion, $\overline{K}$-split, and unipotent (cf.\,Lemma \ref{maximalvec} and Remark \ref{whyred}.1).
\item
$G'''_{i}:=(\Ker(p^{i}\times\colon G\to G))_{\red}$, i.e., the maximal reduced closed subscheme of $\Ker(p^{i}\times\colon G\to G)$ (cf.\,Remark \ref{whyred}.2).
\end{itemize}
\label{overclosed}
\end{thm}

\begin{proof}
First, by Lemmas \ref{fundamentalstructure}.3 and 4, we have $G_{i}=G'_{i}$.

Next, we show that $G'_{i}=G'''_{i}$ holds.
Write $S$ for the closed subset of $C_{\overline{K}}$ consisting of singular points.
For any $s\in S$, write $s_{n_{0}}$ (resp.\,$s_{i}$) for the closed point of $C(n_{0})_{\overline{K}}$ (resp.\,$C(i)_{\overline{K}}$) whose image in $C_{\overline{K}}$ is $s$.
Write $\pi_{i}$ (resp.\,$\pi$) for the natural morphism $C(i)_{\overline{K}}\to C_{\overline{K}}$ (resp.\,$C(n_{0})_{\overline{K}}\to C_{\overline{K}}$).
We have a commutative diagram with exact horizontal lines
\begin{equation}
\xymatrix{
0\ar[r]
&\cO^{\ast}_{C_{\overline{K}}}\ar[r]\ar@{=}[d]
&\pi_{i\ast}\cO^{\ast}_{C(i)_{\overline{K}}}\ar[r]\ar[d]
&\underset{s\in S}{\bigoplus}(\cO_{C(i)_{\overline{K}},s_{i}}^{\ast}/\cO_{C_{\overline{K}},s}^{\ast})\ar[r]\ar[d]
&0\\
0\ar[r]
&\cO^{\ast}_{C_{\overline{K}}}\ar[r]
&\pi_{\ast}\cO^{\ast}_{C(n_{0})_{\overline{K}}}\ar[r]
&\underset{s\in S}{\bigoplus}(\cO_{C(n_{0})_{\overline{K}},s_{n_{0}}}^{\ast}/\cO_{C_{\overline{K}},s}^{\ast})\ar[r]
&0.
}
\label{forPic}
\end{equation}
Here, the sheaves $\cO_{C(n_{0})_{\overline{K}},s_{n_{0}}}^{\ast}/\cO_{C_{\overline{K}},s}^{\ast}$ and $\cO_{C(n_{0})_{\overline{K}},s_{n_{0}}}^{\ast}/\cO_{C(i)_{\overline{K}},s_{i}}^{\ast}$ are skyscraper sheaves whose supports are concentrated on $\{s\}$, from which it follows that the first cohomology groups of these sheaves are trivial.
By taking the long exact sequences of the cohomology groups for (\ref{forPic}), we have the following commutative diagram with exact horizontal lines:
\begin{equation}
\xymatrix{
0\ar[r]
&\underset{s\in S}{\bigoplus}(\cO_{C(i)_{\overline{K}},s_{i}}^{\ast}/\cO_{C_{\overline{K}},s}^{\ast})\ar[r]\ar[d]
&H^{1}(C_{\overline{K}},\cO^{\ast}_{C_{\overline{K}}})\ar[r]\ar@{=}[d]
&H^{1}(C(i)_{\overline{K}},\cO^{\ast}_{C(i)_{\overline{K}}})\ar[r]\ar[d]
&0\\
0\ar[r]
&\underset{s\in S}{\bigoplus}(\cO_{C(n_{0})_{\overline{K}},s_{n_{0}}}^{\ast}/\cO_{C_{\overline{K}},s}^{\ast})\ar[r]
&H^{1}(C_{\overline{K}},\cO^{\ast}_{C_{\overline{K}}})\ar[r]
&H^{1}(C(n_{0})_{\overline{K}},\cO^{\ast}_{C(n_{0})_{\overline{K}}})\ar[r]
&0.
}
\label{cohPic}
\end{equation}
On the other hand, we have the following commutative diagram with exact horizontal lines:
\begin{equation}
\xymatrix{
0\ar[r]
&G'_{i}(\overline{K})\ar[r]\ar[d]
&\Pic_{C_{\overline{K}}/\overline{K}}(\overline{K})\ar[r]\ar@{=}[d]
&\Pic_{C(i)_{\overline{K}}/\overline{K}}(\overline{K})\ar[r]\ar[d]
&0\\
0\ar[r]
&G(\overline{K})\ar[r]
&\Pic_{C_{\overline{K}}/\overline{K}}(\overline{K})\ar[r]
&\Pic_{C(n_{0})_{\overline{K}}/\overline{K}}(\overline{K})\ar[r]
&0.
}
\label{Picexact}
\end{equation}
Hence, we have natural isomorphisms
$$G(\overline{K})=\underset{s\in S}{\bigoplus}(\cO_{C(n_{0})_{\overline{K}},s_{n_{0}}}^{\ast}/\cO_{C_{\overline{K}},s}^{\ast})$$
and
$$G'_{i}(\overline{K})=\underset{s\in S}{\bigoplus}(\cO_{C(n_{0})_{\overline{K}},s_{n_{0}}}^{\ast}/\cO_{C(i)_{\overline{K}},s_{i}}^{\ast}).$$
From these isomorphisms and Lemma \ref{element}, we obtain
$$G'_{i}(\overline{K})=\Ker(p^{i}\times\colon G(\overline{K})\to G(\overline{K}))\simeq\Ker(p^{i}\times\colon G\to G)(\overline{K}).$$
Since $G'_{i}$ is a smooth (and hence reduced) algebraic subgroup of $\Pic_{C_{\overline{K}}/\overline{K}}$ by Proposition \ref{connunip}, we have a natural isomorphism $G'_{i}\simeq G'''_{i}$.

Finally, we see that there are natural inclusions $G'_{i}\hookrightarrow G''_{i}\hookrightarrow G'''_{i}$.
By definition, we have a natural inclusion $G''_{i}\hookrightarrow G'''_{i}$.
The fact that $G'_{i}$ is a connected smooth unipotent algebraic group over $\overline{K}$ follows from Proposition \ref{connunip}.
Then the $\overline{K}$-splitness of $G'_{i}$ follows from {\cite[Expos\'e XVII Proposition 4.1.1]{SGA3}}.
Since $G'_{i}=G'''_{i}$, $G'_{i}$ is $p^{i}$-torsion.
Hence, we have a natural inclusion $G'_{i}\hookrightarrow G''_{i}$.
Therefore, we finish the proof of Theorem \ref{overclosed}.
\end{proof}

\begin{rem}
\begin{enumerate}
\item
Note that, for an algebraic group $H$ over an algebraically closed field, $H$ is $p^{i}$-torsion, split, and unipotent if and only if $p^{i}$-torsion, connected, smooth, and unipotent by {\cite[Expos\'e XVII Proposition 4.1.1]{SGA3}}.
\item
Since the maximal reduced closed subscheme of an algebraic group over a perfect field is again an algebraic group, $G'''_{i}$ is an algebraic group over $\overline{K}$.
Note that $\Ker(p\times\colon G\to G)$ is reduced if and only if $G_{1}=G$.
Indeed, if $G=G_{1}$, $\Ker(p\times\colon G\to G)=G$.
On the other hand, if $G\neq G_{i}$, $pG$ is nontrivial algebraic subgroup of $G$.
Since $G$ and $pG$ are smooth and the natural homomorphism $G\to pG$ induces a trivial homomorphism between their Lie algebras, the algebraic group $\Ker(G\to pG)$ is not smooth.
\end{enumerate}
\label{whyred}
\end{rem}

\begin{cor}
For every natural number $i$, we have a natural homomorphism
$p\times\colon G_{i+2}/G_{i+1}\to G_{i+1}/G_{i}$
inducing an injection $(G_{i+2}/G_{i+1})(\overline{K})\hookrightarrow(G_{i+1}/G_{i})(\overline{K})$.
\label{main1cor}
\end{cor}

\begin{proof}
Since $G_{i+1}$ is connected smooth unipotent $\overline{K}$-split and $p^{i+1}$-torsion by Theorem \ref{overclosed}, $pG_{i+1}$ is connected smooth unipotent $\overline{K}$-split and $p^{i}$-torsion.
Hence, by Theorem \ref{overclosed}, we have $pG_{i+1}\subset G_{i}$ and a homomorphism
$$p\times\colon G_{i+2}/G_{i+1}\to G_{i+1}/G_{i}.$$
Again by Theorem \ref{overclosed}, we have a natural isomorphism
$$(G_{i+1}/G_{i})(\overline{K})\simeq \Ker(p^{i+1}\times\colon G(\overline{K})\to G(\overline{K}))/\Ker(p^{i}\times\colon G(\overline{K})\to G(\overline{K})).$$
Therefore, the induced homomorphism
$$(G_{i+2}/G_{i+1})(\overline{K})\hookrightarrow(G_{i+1}/G_{i})(\overline{K})$$
is injective.
\end{proof}

\begin{rem}
By Corollary \ref{main1cor}, Lemma \ref{fundamentalstructure}, and Theorem \ref{overclosed}, it holds that $\dim G_{i+2}/G_{i+1}\leq\dim  G_{i+1}/G_{i}$, and hence we have $g_{i+1}-g_{i+2}\leq g_{i}-g_{i+1}$.
Note that a stronger inequality $p(g_{i+1}-g_{i+2})\leq g_{i}-g_{i+1}$ follows from Corollary \ref{geqp}.
\label{gpgenuschange}
\end{rem}

\label{main1section}

\section{The case of simple extensions} 
In this section, we give calculations of conductors and $\delta$-invariants.
Here, $\delta$-invariants can be regarded as certain local variants of ``genus changes'' of curves.

Let $K$ be a field of characteristic $p$, $R$ a discrete valuation ring over $K$, $\fm$ the maximal ideal of $R$, $\varpi$ a uniformizer of $R$, $v$ the normalized valuation of $R$, and $L$ the residue field of $R$.
Let $x$ be an element of $K$ satisfying $x\in K\setminus K^{p}$.
Suppose that $R$ is $\{x\}$-integral.
Write $R(1_{x})$ for the normalization of $R\otimes_{K}K(x^{1/p})$ in its field of fractions, which is a discrete valuation ring.
(Note that, in Section \ref{geomnormalsection}, we write $R(1_{\{x\}})$ for $R(1_{x})$.)
Write $v_{1}$ (resp.\,$L(1_{x})$; $\fm(1_{x})$; $\varpi_{1}$; $e_{1}$; $f_{1}$) for the valuation of $R(1_{x})$ (resp.\,the residue field of $R(1_{x})$; the maximal ideal of $R(1_{x})$; a uniformizer of $R(1_{x})$; the ramification index of $R(1_{x})$ over $R$; the residue degree of $R(1_{x})$ over $R$).
Suppose that $R(1_{x})$ is finite over $R$ (or equivalently, $\{x\}$-Japanese, cf.\,Proposition \ref{B-prop}.5).
Then we have $e_{1}f_{1}=p$ by {\cite[I \S4 Proposition 10]{Ser}}.
Moreover, note that $R, R\otimes_{K}K(x^{1/p})$, and $R(1_{x})$ are local Noetherian domains of dimension $1$.
For such a local domain $A$, write $\widehat{A}$ for the completion of $A$.
Since $R(1_{x})$ is finite over $R$, we have $\widehat{A}\simeq A\otimes_{R}\widehat{R}$.
Since $\widehat{R}$ is a flat over $R$, all such $\widehat{A}$ are local domains.

\begin{not-dfn}
\begin{enumerate}
\item 
We write $\delta_{10}$ for the natural number $$\length_{R\otimes_{K}K(x^{1/p})}(R(1_{x})/(R\otimes_{K}K(x^{1/p}))),$$
and refer to this as the {\it $\delta$-invariant} (of $R\otimes_{K}K(x^{1/p})$) (cf.\,\cite[0C3Q]{Stacks}).
\item
If $L$ is finite over $K$, we write
$g_{10}$ for the natural number $$\dim_{K(x^{1/p})}(R(1_{x})/(R\otimes_{K}K(x^{1/p}))).$$
\item
We write $\cC_{10}$ for the largest ideal of $R(1_{x})$ contained in $R\otimes_{K}K(x^{1/p})$ and refer to this as {\it the conductor} (cf.\,{\cite[III \S6]{Ser}} and {\cite[Definition 2.13]{PW}}).
\end{enumerate}
\label{deltaconddef}
\end{not-dfn}

\begin{rem}
\label{remdelta}
In \cite[Section 5]{IIL}, the definition of the $\delta$-invariant is a little different.
In their definition, they replace $K(x^{1/p})$ in our definition with an algebraic closure of $K$.
Since we want to consider the step-by-step behavior of invariants, we follow the convention in \cite[0C3Q]{Stacks}.
\end{rem}

\begin{lem}
\begin{enumerate}
\item
If $x^{1/p}\notin L$, then $R(1_{x})=R\otimes_{K}K(x^{1/p})$.
Moreover, $R$ is $\{x\}$-normal.
\item
$\delta_{10}=\length_{R}(R(1_{x})/(R\otimes_{K}K(x^{1/p})))$
\item
If $L$ is finite over $K$ and $x^{1/p}\in L$, we have
$$g_{10}=\delta_{10}[L:K(x^{1/p})].$$
\end{enumerate}
\label{calcul}
\end{lem}

\begin{proof}
Assertion 1 follows from {\cite[I \S6 Proposition 15]{Ser}} and Proposition \ref{B-prop}.2.
To show assertion 2, we may assume $x^{1/p}\in L$ by assertion 1.
Since the residue field of $R\otimes_{K}K(x^{1/p})$ is naturally isomorphic to $L$, we have assertion 2.
Moreover, if $L$ is finite over $K$, then we have
\begin{align*}
[L:K(x^{1/p})] &\length_{R\otimes_{K}K(x^{1/p})}(R(1_{x})/(R\otimes_{K}K(x^{1/p}))) \\
= &\dim_{K(x^{1/p})}(R(1_{x})/(R\otimes_{K}K(x^{1/p}))).
\end{align*}
Hence, assertion 3 holds.
\end{proof}

\begin{nota}
For any lift $r\in R$ of $x^{1/p}$, we shall write $q_{r}$ for $v(r^{p}-x)$.
Moreover, we shall write
$$q(x)=\sup_{r\in R}v(r^{p}-x).$$
Note that if $x^{1/p}\in L$, $q(x)$ coincides with $\underset{r}{\sup}\,q_{r}$, where $r$ ranges over all the lifts of $x^{1/p}$.
Moreover, note that if $x^{1/p}\notin L$, $q(x)=0$.
\label{notationq}
\end{nota}
In the following, suppose that $x^{1/p}\in L$.
Note that $pv_{1}(r-x^{1/p})=v_{1}(r^{p}-x)=e_{1}v(r^{p}-x)=e_{1}q_{r}$.
The following is the first main theorem of this section.
\begin{thm}
\begin{enumerate}
\item
The inequalities $1\leq q(x)<\infty$ holds (cf.\,Remark \ref{metricinterprete}).
\item
$e_{1}=p$ if and only if $q(x)$ is not divisible by $p$.
\item
If $e_{1}=p$, we have
\begin{align*}
\cC_{10}=\fm(1_{x})^{(p-1)(q(x)-1)}\\
\delta_{10}=\frac{(p-1)(q(x)-1)}{2}.
\end{align*}
In particular, $R$ is $\{x\}$-normal if and only if $q(x)=1$.
\item
If $f_{1}=p$, we have
\begin{eqnarray*}
\cC_{10}=\fm(1_{x})^{\frac{(p-1)q(x)}{p}}\\
\delta_{10}=\frac{(p-1)q(x)}{2}.
\end{eqnarray*}
In particular, if $p\neq 2$, then $\delta_{10}$ is divisible by $p$.
\item
We have
$$q(x)=\max_{r'\in R, x'\in K^{p}(x)\setminus K^{p}} v(r'^{p}-x')=\max_{r''\in R^{p}, x'\in K^{p}(x)\setminus K^{p}}v(r''-x').$$
In particular, for any element $x'\in K^{p}(x)\setminus K^{p}$, we have $q(x')=q(x)$ (cf.\,Remark \ref{metricinterprete}).
\end{enumerate}
\label{q}
\end{thm}

\begin{rem}
There exists an element of $R$ which is one of the nearest elements to $x$ with respect to the metric of $R(1_{x})$, which attains the distance between $R^{p}$ and $K^{p}(x)\setminus K^{p}(=K^{p}(x)\setminus R^{p})$.
\label{metricinterprete}
\end{rem}

\begin{rem}
\label{conductorvsgenuschange}
From Theorems \ref{q}.2 and 3, we obtain a relation
\begin{align*}
2\length_{R}(R(1_{x})/R\otimes_{K}K(x^{1/p}))
&=2\delta_{10}\\
&=[L(1_{x}):L]\length_{R(1_{x})}(R(1_{x})/\cC_{10})\\
&=\length_{R}(R(1_{x})/\cC_{10})
\end{align*}
of the conductor and the $\delta$-invariant.
As we will see below, this equation follows from the theory of duality of modules.
Let $A\subset B$ be an extension of $1$-dimensional Gorenstein local domains such that $B/A$ is an $A$-module of finite length.
We will show that
\begin{align}
\label{2times}
2 \length_{A}(B/A) = \length_{A} (B/\cC),
\end{align}
where $\cC$ is the image of 
\[
\omega_{B/A}
= \Hom_{A}(B, A)
\hookrightarrow A
\]
via the trace map.
Note that, in the setting of Notation-Definition \ref{deltaconddef}, the extension $R\otimes_{K}(K(x^{1/p}))\subset R(1_{x})$ satisfies the assumptions on $A\subset B$.
We also note that, in this case, $\cC$ is no other than $\cC_{10}$, and $\length_{R}$ and $\length_{A}$ have the same value for $A$-modules.
To show that (\ref{2times}) holds, it suffices to show that 
\[
\length_{A}(A/\cC) = \length_{A}(B/A),
\]
since we have $(B/\cC)/(A/\cC)=B/A$.
By applying $R\Hom_{A}(-,A)$ to the exact sequence
\[
0 \rightarrow A \rightarrow B \rightarrow B/A \rightarrow 0,
\]
we obtain an exact sequence
\[
0 \rightarrow \Hom_{A}(B,A) \rightarrow \Hom_{A}(A,A) \rightarrow \Ext^{1}_{A}(B/A,A) \rightarrow 0,
\]
i.e.,\,we have $A/\cC\simeq\Ext^{1}_{A}(B/A,A)$.
By the local duality,
\[
\Ext^{1}_{A}(B/A,A) \simeq \Hom_{A}(B/A,E),
\]
where $E$ is an injective hull of the residue field of $A$.
Since the Matlis duality preserves the length of modules, we obtain the desired equality.
\end{rem}

\begin{proof}[Proof of Theorem \ref{q}.1]
Since $x^{1/p}\in L$, we have $1 \leq q(x)$.
Next we show that $q(x)<\infty$.
Note that, since $R$ is geometrically reduced over $K$, we have $x^{1/p}\notin R$.
Since $\widehat{R}$ is faithfully flat over $R$, we also have $x^{1/p}\notin\widehat{R}$.
Therefore, since $\widehat{R}$ is a closed subset of $\widehat{R(1_{x})}$, there exists an element $r$ of $R$ which satisfies $q_{r}=q(x)$.
\end{proof}

To show the rest of Theorem \ref{q}, we need some lemmas.

\begin{lem}
For any lift $r\in R$ of $x^{1/p}$, $q_{r}<q(x)$ if and only if $p$ divides $q_{r}$ and the image of $r-x^{1/p}$ in $\fm(1_{x})^{e_{1}q_{r}/p}/\fm(1_{x})^{(e_{1}q_{r}/p)+1}$ is contained in $\fm^{q_{r}/p}/\fm^{(q_{r}/p)+1}$.
\label{q<q}
\end{lem}

\begin{proof}
By the definition of $q(x)$, $q_{r}<q(x)$ if and only if there exists an element $r'\in R$ such that 
\begin{align}
v(r'^{p}+r^{p}-x)>v(r^{p}-x)(\Leftrightarrow v_{1}(r'+r-x^{1/p})>v_{1}(r-x^{1/p})).
\label{forq<q}
\end{align}
Since the inequality (\ref{forq<q}) holds only if $v(r'^{p})=v(r^{p}-x)$, we may assume that $p$ divides $q_{r}$.
Then we have a natural injection $$\fm^{q_{r}/p}/\fm^{(q_{r}/p)+1}\hookrightarrow \fm(1_{x})^{e_{1}q_{r}/p}/\fm(1_{x})^{(e_{1}q_{r}/p)+1}.$$
Then there exists $r'$ satisfying the inequality (\ref{forq<q}) if and only if the image of $r-x^{1/p}$ in $\fm(1_{x})^{e_{1}q_{r}/p}/\fm(1_{x})^{(e_{1}q_{r}/p)+1}$ is contained in $\fm^{q_{r}/p}/\fm^{(q_{r}/p)+1}$.
Indeed, if such an element $r'$ exists, then $r-x^{1/p}= -r' + (r'+r-x^{1/p})$ is contained in $\fm^{q_{r}/p}/\fm^{(q_{r}/p)+1}$, and the converse also holds in a similar way.
%Conversely, if $r-x^{1/p}$ is written as $s + ((r-x^{1/p})-s)$, where $s\in \fm^{q_{r}/p} \subset R$ and $(r-x^{1/p})-s \in \fm(1_{x})^{(e_{1}q_{r}/p)+1}$, then the inequality 
Hence, Lemma \ref{q<q} holds.
\end{proof}

\begin{lem}
$e_{1}=p$ if and only if there exists a lift $t\in R$ of $x^{1/p}$ such that $q_{t}$ is not divisible by $p$.
In this case, such $t$ satisfies $q(x)=q_{t}$.
In particular, $e_{1}=p$ if and only if $q(x)$ is not divisible by $p$.
\label{r-x}
\end{lem}

\begin{proof}
If $e_{1}=1$, for any lift $r\in R$ of $x^{1/p}$, $q_{r}=pv_{1}(r-x^{1/p})$ is divisible by $p$.
In particular, in this case, $q(x)$ is divisible by $p$.
If $e_{1}=p$, we have a natural isomorphism $\fm^{pi}/\fm^{pi+1}\simeq \fm(1_{x})^{i}/\fm(1_{x})^{i+1}$ for every $i\in \N$.
In this case, for any lift $r\in R$ of $x^{1/p}$, $q_{r}$ is not divisible by $p$ if and only if $q_{r}=q(x)$ by Lemma \ref{q<q}.
 (Note that such $r$ exists by Theorem \ref{q}.1.)
\end{proof}

\begin{proof}[Proof of Theorems \ref{q}.2 and \ref{q}.3]
Theorem \ref{q}.2 follows from Lemma \ref{r-x}.

Suppose that $e_{1}=p$ and take an element $s\in R$ satisfying the condition on $t$ in Lemma \ref{r-x}.
Then $\varpi_{1}^{p}$ is a uniformizer of $R$.
Note that $R\otimes_{K}K(x^{1/p})=R[x^{1/p}]=R[s-x^{1/p}]$.
Since the characteristic of the field of fractions of $\widehat{R}$ is $p$, there exists a subfield $L'$ of $\widehat{R}$ such that the composite homomorphism $L'\hookrightarrow \widehat{R}\twoheadrightarrow L$ is an isomorphism by the Cohen structure theorem.
Then $\widehat{R}\hookrightarrow (R\otimes_{K}K(x^{1/p}))^{\wedge}\hookrightarrow \widehat{R(1_{x})}$ can be identified with $L'[[\varpi_{1}^{p}]]\hookrightarrow L'[[\varpi_{1}^{p}, s-x^{1/p}]]\hookrightarrow L'[[\varpi_{1}]]$ by {\cite[I \S6 Proposition 18]{Ser}}.
(Here, $(R\otimes_{K}K(x^{1/p}))^{\wedge}$ denotes the completion of $R\otimes_{K}K(x^{1/p})$.)
Then we have
\begin{equation*}
\begin{split}
\delta_{10}&=\length_{R}(R(1_{x})/(R\otimes_{K}K(x^{1/p})))\\
&=\length_{\widehat{R}}(\widehat{R(1_{x})}/(R\otimes_{K}K(x^{1/p}))^{\wedge})\\
&=\dim_{L'}(L'[[\varpi_{1}]]/L'[[\varpi_{1}^{p}, s-x^{1/p}]])\\
&=\frac{(p-1)(q(x)-1)}{2}
\end{split}
\end{equation*}
and
$$\cC_{10}=\fm(1_{x})^{(p-1)(q(x)-1)}$$
by Lemma \ref{r-x} and Lemma \ref{coin}.
\end{proof}

\begin{lem}
Let $M$ be a field, $T$ a variable, and $m, n$ positive integers.
Suppose that $m$ and $n$ are coprime.
Let $M[[T]]$ be the ring of formal power series and $\gamma(T), \delta(T)\in M[[T]]$ units.
Then we have
$$\dim_{M}(M[[T]]/M[[T^{m}\gamma(T), T^{n}\delta(T)]])=\frac{(m-1)(n-1)}{2}$$
and the conductor of $M[[T^{m}\gamma(T), T^{n}\delta(T)]]$ in $M[[T]]$ is $(T^{(m-1)(n-1)})$.
\label{coin}
\end{lem}

\begin{proof}
In the case where $\gamma(T)=\delta(T)=1$, this lemma is the so-called ``Frobenius coin problem".
In this case, there exist natural numbers $\alpha_{1},\ldots,\alpha_{\frac{(m-1)(n-1)}{2}}$ such that
$$M[[T]]=M[[T^{m},T^{n}]]\oplus\bigoplus_{1\leq i\leq \frac{(m-1)(n-1)}{2}}MT^{\alpha_{i}},$$
and, for any $\beta\geq (m-1)(n-1)$, there also exists polynomials $F_{\beta}(X,Y)\in M[X,Y]$  such that
$$T^{\beta}=F_{\beta}(T^{m},T^{n}).$$

Next, we prove Lemma \ref{coin} for general $\gamma(T)$ and $\delta(T)$.
We may assume that the constant terms of $\gamma(T)$ and $\delta(T)$ are $1$.
It follows that
$$M[[T]]=M[[T^{m}\gamma(T),T^{n}\delta(T)]]\oplus\bigoplus_{1\leq i\leq \frac{(m-1)(n-1)}{2}}MT^{\alpha_{i}}$$
from induction argument and the fact that, for any natural number $l<nm$, there are at most $1$ pair $(i,j)$ of natural numbers satisfying $ni+mj=l$.
Moreover, for any natural number $\beta\geq(m-1)(n-1)$, we have
$$T^{\beta}= F_{\beta}(T^{m}\gamma(T),T^{n}\delta(T)) \in \fm(1_{x})^{\beta}/\fm(1_{x})^{\beta+1}.$$
Again by induction argument, we can show that the conductor of $M[[T^{m}\gamma(T), T^{n}\delta(T)]]$ in $M[[T]]$ is $(T^{(m-1)(n-1)})$.
\end{proof}

\begin{lem}
Suppose that $f_{1}=p$.
Let $r\in R$ be a lift of $x^{1/p}$.
For such $r$, we write $q'_{r}$ for $q_{r}/p$, $u_{r}$ for the element $(r-x^{1/p})/(\varpi^{q'_{r}})$ in $R(1_x)$, and $\overline{u_{r}}$ for the image of $u_{r}$ in $L(1_{x})$.
The following are equivalent:
\begin{enumerate}
\item
$R[u_{r}]=R(1_{x})$.
\item
$L[\overline{u_{r}}]=L(1_{x})$.
\item
$q_{r}=q(x)$.
\end{enumerate}
\label{choiceofr}
\end{lem}

\begin{proof}
By {\cite[I \S6 Proposition 15]{Ser}}, $R(1_{x})=R[u_{r}]$ if and only if $L(1_{x})=L(\overline{u_{r}})$.
Since we have $L\varpi^{q'_{r}}=\fm^{q'_{r}}/\fm^{q'_{r}+1}\subset \fm(1_{x})^{q'_{r}}/\fm(1_{x})^{q'_{r}+1}$, $\overline{u_{r}}\in L$ if and only if the image of $r-x^{1/p}$ in $\fm(1_{x})^{q'_{r}}/\fm(1_{x})^{q'_{r}+1}$ is contained in $\fm^{q'_{r}}/\fm^{q'_{r}+1}$, which is equivalent to $q_{r}<q(x)$ by Lemma \ref{q<q}. 
\end{proof}

\begin{proof}[Proof of Theorems \ref{q}.4 and \ref{q}.5]
We keep the notation and the assumption of Lemma \ref{choiceofr}.
Let $s\in R$ be a lift of $x^{1/p}$ satisfying $q(x)=q_{s}$ (cf.\,Theorem \ref{q}.1).
Write $q'(x)$ for $q(x)/p$.
We have $R\otimes_{K}K(x^{1/p})=R[u_{s}\varpi^{q'(x)}]$.
Since $s-x^{1/p}=\varpi^{q'(x)}u_{s}$ and $R\otimes_{K}K(x^{1/p})=R[x^{1/p}]$, we have $R\otimes_{K}K(x^{1/p})=R[u_{s}\varpi^{q'(x)}]$.
Note that we have $R(1_{x})=R[u_{s}]=\sum_{0\leq i\leq p-1}Ru_{s}^{i}$.
Then we have
$$R(1_{x})\varpi^{pq'(x)}=\sum_{0\leq i\leq p-1}Ru_{s}^{i}\varpi^{pq'(x)}=\sum_{0\leq i\leq p-1}R(s-x)^{i}\varpi^{(p-i)q'(x)}\subset R\otimes_{K}K(x^{1/p}).$$
Therefore, we have
\begin{equation*}
\begin{split}
\delta_{10}&=\length_{R}(R(1_{x})/R\otimes_{K}K(x^{1/p}))\\
&=\length_{R}(R(1_{x})/R[u_{s}]\varpi^{pq'(x)})/(R[u_{s}\varpi^{q'(x)}]/R[u_{s}]\varpi^{pq'(x)}).
\end{split}
\end{equation*}
Write $\overline{\fm(1_{x})}$ for the image of $\fm(1_{x})$ in $R(1_{x})/R[u_{s}]\varpi^{pq'(x)}$.
Then, for every natural number $j$, we have
$$\overline{\fm(1_{x})}^{j}/\overline{\fm(1_{x})}^{j+1}=\underset{\substack{1\leq i\leq p-1\\j<pq'(x)i}}{\bigoplus}L u_{s}^{i}\varpi^{j}.$$
Therefore, we have
$$\delta_{10}=\frac{p(p-1)q'(x)}{2}=\frac{(p-1)q(x)}{2},\quad
\cC_{10}=\fm(1_{x})^{pq'(x)-q'(x)}=\fm(1_{x})^{\frac{(p-1)q(x)}{p}}.$$

Theorem \ref{q}.5 follows from Theorems \ref{q}.1, \ref{q}.3, and \ref{q}.4.
\end{proof}

Suppose that $x^{1/p^{2}}\notin L(1_{x})$.
By replacing $R$ and $K$ with $R(1_{x})$ and $K(x^{1/p})$, respectively, we define $R(2_{x}), e_{2}, f_{2}, g_{21}, \delta_{21}$, and $q(x^{1/p})$ in a similar way.
Note that $R(2_{x})$ is finite over $R(1_{x})$ since $R$ is $\{x\}$-Japanese.

\begin{lem}
Let $r_{1}\in R(1_{x})$ be a lift of $x^{1/p^{2}}\in L(1_{x})$ satisfying $v_{1}(r_{1}^{p}-x^{1/p})=q(x^{1/p})$.
Note that $r_{1}^{p}\in R$ is a lift of $x^{1/p}$.
\begin{enumerate}
\item
Suppose that $e_{1}=e_{2}=p$.
We have $q(x^{1/p})=q_{r_{1}^{p}}=q(x)$ and $\delta_{21}=\delta_{10}$.
\item
Suppose that $e_{1}=p$ and $e_{2}=1$.
We have $q(x^{1/p})=q_{r_{1}^{p}}< q(x)$ and $\delta_{21}\leq \delta_{10}$. 
\item
Suppose that $e_{1}=e_{2}=1$.
We have $pq(x^{1/p})=q_{r_{1}^{p}}\leq q(x)$ and $p\delta_{21}\leq \delta_{10}$. 
Moreover, the equality $pq(x^{1/p})=q(x)$ holds if and only if the equality $p\delta_{21}=\delta_{10}$ holds.
Furthermore, $L(2_{x})$ is a simple extension field over $L$ if and only if $pq(x^{1/p})=q(x)$.
\item
Suppose that $e_{1}=1$ and $e_{2}=p$.
We have $pq(x^{1/p})=q_{r_{1}^{p}}\leq q(x)$ and $p\delta_{21}+\frac{p(p-1)}{2}\leq \delta_{10}$. 
Moreover, the equality $p q(x^{1/p})= q(x)$ holds if and only if the equality $p\delta_{21}+\frac{p(p-1)}{2} = \delta_{10}$ holds.
\end{enumerate}
\label{qlem}
\end{lem}

Before we prove Lemma \ref{qlem}, we state one of the main theorems of this paper (cf.\,Theorem \ref{intromain2}), which follows immediately from this lemma.

\begin{thm}
We keep the notation of Lemma \ref{qlem}.
In each assertion $1\leq i \leq 4$ of this theorem, we assume the same assumption as that in assertion $i$ of Lemma \ref{qlem}. 
Suppose that $L$ is finite over $K$.
The we have the following:
\begin{enumerate}
\item
$pg_{21}=g_{10}$.
\item
$pg_{21}\leq g_{10}$.
\item
$pg_{21}\leq g_{10}$.
Moreover, the equality holds if $L(2_{x})$ is a simple field extension over $L$.
\item
$pg_{21}+\frac{p(p-1)}{2}\leq g_{10}$.
\end{enumerate}
\label{mainineq}
\end{thm}

\begin{proof}[Proof of Lemma \ref{qlem}]
First, we have an inequality $q_{r_{1}^{p}}\leq q(x)$ by the definition of $q(x)$.

We show assertions 1 and 2.
Suppose that $e_{1}=p$.
Then we have
$$q(x^{1/p})=v_{1}(r_{1}^{p}-x^{1/p})=v((r_{1}^{p})^{p}-x)=q_{r_{1}^{p}}.$$
If $e_{2}=p$, since $q_{r_{1}^{p}}=q(x^{1/p})$ is not divisible by $p$ by Lemma \ref{r-x}, $q_{r_{1}^{p}}=q(x)$ again by Lemma \ref{r-x}.
Hence, assertion 1 holds.
If $e_{2}=1$, since $q(x)$ is divisible by $p$ and $q(x^{1/p})$ is not divisible by $p$, we have $q(x^{1/p})<q(x)$.

Next, we show assertions 3 and 4.
Suppose that $e_{1}=1$.
Then we have
$$pq(x^{1/p})=pv_{1}(r_{1}^{p}-x^{1/p})=v((r_{1}^{p})^{p}-x)=q_{r_{1}^{p}}.$$
If $e_{2}=p$, assertion 4 follows from Theorems \ref{q}.2, \ref{q}.3, and \ref{q}.4.
Suppose that $e_{2}=1$.
Write $u'$ (resp.\,$\overline{u'}$) for the element 
\[
\frac{r_{1}^{p}-x^{1/p}}{\varpi^{q'(x^{1/p})}}\in R(2_{x})
\] 
(resp.\,the image of $u'$ in $L(2_{x})$).
Since $e_{1}=1$, $\varpi$ is a uniformizer of $R(1_{x})$ and $L(2_{x})=L(1_{x})[\overline{u'}]$ by Lemma \ref{choiceofr}.
Moreover, we have
$$u'^{p}=(\frac{r_{1}^{p}-x^{1/p}}{\varpi^{q'(x^{1/p})}})^{p}=\frac{(r_{1}^{p})^{p}-x}{\varpi^{pq'(x^{1/p})}}.$$
Then, by Lemma \ref{choiceofr}, the following are equivalent:
\begin{itemize}
\item
$L(2_{x})$ is a simple field extension over $L$.
\item
$L(1_{x})=L[\overline{u'}^{p}]$.
\item
$q_{r_{1}^{p}}=q(x)$.
\end{itemize}
Therefore, assertion 3 follows from Theorems \ref{q}.2 and \ref{q}.4.
\end{proof}

\begin{cor}
Suppose that $e_{1}=e_{2}=p$ and $R(1_{x})$ is $\{x^{1/p}\}$-normal.
Then $R$ is $\{x\}$-normal.
\end{cor}

\begin{proof}
This follows from Theorem \ref{q}.3 and Lemma \ref{qlem}.1.
\end{proof}

By using the above results, we can treat the case of general extensions as follows.

\begin{not-set}
Let $K, R, \fm, \varpi, v$, and $L$ be as above.
We suppose that $L$ is a finite extension of $K$.
Let $B$ be a subset of a $p$-basis over $K$.
Suppose that $R$ is $B$-integral and $B$-Japanese.
We write $R(1_{B})$ for the normalization of $R\otimes_{K}K(B^{1/p})$ in its field of fractions as before.
We put $g_{10}':= \dim_{K(B^{1/p})}(R(1_{B})/R\otimes_{K}K(B^{1/p}))$.
By replacing $K, R,$ and $B$ with $K(B^{1/p}), R(1_{B}),$ and $B^{1/p}$, respectively, we define $R(2_{B})$ and $g'_{21}$ in a similar way.
We note that $g_{10}',g'_{21} < \infty$ since $R$ is $B$-Japanese and $L$ is finite over $K$.
\label{notageneral}
\end{not-set}

\begin{prop}
In the setting of Notation-Setting \ref{notageneral},
we have $p g_{21}' \leq g_{10}'$.
\label{Bgeqp}
\end{prop}
\begin{proof}
For each $i \in \{ 1,2\}$, there exist elements $r_{i1}, \ldots, r_{is_{i}} \in R(i_{B})$ such that 
\[
R(i_{B})= \sum_{j}(R\otimes_{K}K(B^{1/p^{i}}))r_{ij}.
\]
Then we can take a finite subset $B' \subset B$ such that $r_{ij} \in \Frac (R\otimes_{K}K(B'^{1/p^{i}}))$ for any $i,j$.
Then the problem can be reduced to the case for $B'$, so we may suppose that $B$ is a finite set.
We write $B = \{x_{1}, \ldots, x_{m}\}.$
We consider extensions $K \subset K(x_{1}^{1/p}) \subset K(B^{1/p}) \subset K(B^{1/p})(x_{1}^{1/p^{2}})$.
We put 
\[
g_{10}^{x_{1}} := \dim_{K(x_{1}^{1/p})} (R(1_{x_{1}})/ R \otimes_{K} K (x_{1}^{1/p}))
\]
and
\[
g_{21}^{x_{1}} := \dim_{K(B^{1/p})(x_{1}^{1/p^{2}})} (R(1_{B})(1_{x_{1}^{1/p}})/ R(1_{B})\otimes_{K(B^{1/p})} K (B^{1/p})(x_{1}^{1/p^{2}}) ).
\]
It suffices to show that $p g_{21}^{x_{1}} \leq g_{10}^{x_{1}}$.
Indeed, by using the same inequality with respect to $x_{2}$ for the normalization of $R\otimes_{K} K(x_{1}^{1/p})$, and iterating these calculations, we obtain the desired inequality.

Let $v'$ be the normalized valuation of $R(1_{B})$ and $L'$ the residue field of $R(1_{B})$.
We put $v'|_{R}= p^{i}v$ and $[L':L] = p^{j}$.
We note that $i$ is $1$ or $0$, and $i+j=m$.
We may assume that $x_{1}^{1/p^{2}} \in L'$.
Otherwise, we have $g_{21}^{x_{1}} = 0$ by \cite[I \S6 Proposition 15]{Ser}.
Therefore, we can take $r\in R$ such that $q(x_{1}) = q_{r}$ as in Notation \ref{notationq}.
Moreover, we can take $r' \in R(1_{B})$ such that $q(x_{1}^{1/p}) = q_{r'}$.
We have
\begin{align*}
q_{r'} &= v' (r'^{p}- x_{1}^{1/p})\\
   &= p^{i-1} v (r'^{p^2} - x_{1})\\
   &\leq p^{i-1} v(r^{p}- x_{1}) = p^{i-1} q_{r}.
\end{align*}
First, we consider the case where $p \nmid q_{r'}$.
In this case, we have
\begin{align*}
g_{21}^{x_{1}} &=
\frac{p-1}{2} (q_{r'}-1) [L': K(B^{1/p})(x_{1}^{1/p^2})] \\
&\leq \frac{p-1}{2} (p^{i-1}q_{r}-1) [L:K(x_{1}^{1/p})] p^{j-m} \\
&\leq \frac{p-1}{2} p^{i-1} (q_{r}-1) [L:K(x_{1}^{1/p})] p^{j-m}  \leq \frac{1}{p} g_{10}^{x_{1}}.
\end{align*}
Next, we consider the case where $p \mid q_{r'}$.
If $p \mid q_{r}$, then we have
\begin{align*}
g_{21}^{x_{1}}
&= \frac{p-1}{2}q_{r'}[L':K(B^{1/p})(x_{1}^{1/p^{2}})]\\
& \leq \frac{p-1}{2} p^{i-1} q_{r} [L:K(x_{1}^{1/p})]p^{j-m}
=\frac{1}{p} g_{10}^{x_{1}}.
\end{align*}
On the other hand, if $p \nmid q_{r}$, then we have $i=1$, and the inequality
\begin{align*}
g_{21}^{x_{1}} &= \frac{p-1}{2}q_{r'}[L':K(B^{1/p})(x_{1}^{1/p^{2}})]\\
& \leq \frac{p-1}{2} (q_{r}-1) [L:K(x_{1}^{1/p})]\frac{p^{m-1}}{p^{m}}
=\frac{1}{p} g_{10}^{x_{1}}.
\end{align*}
follows from $q_{r'} \leq q_{r}-1$.
Therefore, we have $p g_{21}^{x_{1}} \leq g_{10}^{x_{1}}$ in any case, and we finish the proof of Proposition \ref{Bgeqp}.
\end{proof}

\begin{not-set}
Let $\overline{K}, C$, and $G$ be as in Section \ref{main1section}.
Let $c$ be a closed point of $C$.
Write $C_{K(x^{1/p})}$ (resp.\,$C(1_{x})$) for the scheme $C\times_{\Spec K}\Spec K(x^{1/p})$ (resp.\,the normalization of $C_{K(x^{1/p})}$ in its function field).
We use the notation of this section assuming $R=O_{C,c}$, and we suppose $x^{1/p}\in L$.
Let $G_{c}$ be the algebraic group over $K(x^{1/p})$ which represents the functor defined by $R(1_{x})^{\ast}/(R\otimes_{K}K(x^{1/p}))^{\ast}$ (cf.\,the proofs of Lemmas \ref{connsmunip}.3 and \ref{connsmunip}.4).
By the proofs of Lemma \ref{connsmunip}.2 and Proposition \ref{connunip}, $G_{c}$ is a connected unipotent algebraic group.
Moreover, $G_{c}$ is $p$-torsion by Theorem \ref{overclosed}.
\label{forKnot}
\end{not-set}

\begin{prop}
We work under the setting of Notation-Setting \ref{forKnot}.
Let $G_{c,v}$ (resp.\,$G_{c,s}$) be the largest vector subgroup (resp.\,the largest $K(x^{1/p})$-split algebraic subgroup) of $G_{c}$ (cf.\,the proof of Lemma \ref{maximalvec}).
(Note that $G_{c,v}$ coincides with the largest $p$-torsion $K(x^{1/p})$-split unipotent subgroup of $G_{c}$ by Lemma \ref{splitvector}.)
Then we have $G_{c,v}=G_{c,s}$.
Moreover, the following are equivalent:
\begin{enumerate}
\item
$G_{c,v}(=G_{c,s})=G_{c}$.
\item
$e_{1}=p$.
\end{enumerate}
\label{pointsubgroup}
\end{prop}

\begin{proof}
Since $G_{c}$ is $p$-torsion and unipotent, the desired equality $G_{c,v}=G_{c,s}$ holds.
Next, we show the equivalence between $1$ and $2$.
By Lemmas \ref{connsmunip}.2, \ref{connsmunip}.3, and \ref{connsmunip}.4 and the proof of Proposition \ref{connunip}, it suffices to show the fact that $e_{1}=p$ if and only if the algebraic group $\Coker((\G_{m})_{L/K}\to(\G_{m})_{L(1_{x})/K})$ is $K(x^{1/p})$-split.
Since we have
\begin{equation*}
\begin{split}
&\Coker((\G_{m})_{L/K}\to(\G_{m})_{L(1_{x})/K})\times_{\Spec K}\Spec L\\
\simeq\,&\Coker((\G_{m})_{L/K}\times_{\Spec K}\Spec L\to(\G_{m})_{L(1_{x})/K}\times_{\Spec K}\Spec L),
\end{split}
\end{equation*}
these algebraic groups are $L$-split if and only if $L=L(1_{x})$ by Lemma \ref{forsplit}, 
\cite[Exercises 14.3.12. (2)]{Sp}, and Lemma \ref{Weilstr}.
Therefore, the desired equivalence holds.
\end{proof}

\begin{thm}
We work under the setting of Notation-Setting \ref{forKnot}.
We define the following four algebraic subgroups of the algebraic group $\Pic^{0}_{C_{K(x^{1/p})}/K(x^{1/p})}$:
\begin{itemize}
\item
The smallest connected linear algebraic subgroup $G_{x}$ such that the quotient algebraic group $\Pic^{0}_{C_{K(x^{1/p})}/K(x^{1/p})}/G_{x}$ is an abelian variety over $K(x^{1/p})$ (cf.\,\cite[Theorem 9.2.1]{BLR}).
\item
$G_{x}(1):=\Ker(\Pic^{0}_{C_{K(x^{1/p})}/K(x^{1/p})}\to\Pic^{0}_{C(1_{x})/K(x^{1/p})})$.
\item
The largest vector subgroup $G_{x,v}$ (cf.\,Lemma \ref{splitvector} and Lemma \ref{maximalvec}).
\item
The largest $K(x^{1/p})$-split algebraic subgroup $G_{x,s}$ (cf.\,the proof of Lemma \ref{maximalvec}).
\end{itemize}
Then we have
$$G_{x,v}=G_{x,s}\subset G_{x}(1)\subset G_{x}.$$
Moreover, the following are equivalent:
\begin{enumerate}
\item
$G_{x,v}=G_{x,s}=G_{x}(1).$
\item
For any singular point $c$ of $C$, the ramification index ``$e_{1}$'' for $R=O_{C,c}$ is equal to $p$.
\end{enumerate}
\label{x01cor}
\end{thm}

\begin{proof}
First, we have $G_{x,v}\subset G_{x,s}$.
By Lemma \ref{splitvector} and Theorem \ref{overclosed}, we have a natural injective homomorphism
$$(G_{x,s}/G_{x,v})\times_{\Spec K(x^{1/p})}\Spec K^{1/p}\hookrightarrow \Pic^{0}_{C(1)/K^{1/p}}.$$
Since $G_{x,s}/G_{x,v}$ is $K(x^{1/p})$-split by \cite[Exercises 14.3.12. (2)]{Sp} and $\Pic^{0}_{C(1)/K^{1/p}}$ does not contain algebraic groups isomorphic to $\G_{a,K^{1/p}}$ by {\cite[Proposition 9.2.4]{BLR}}, $G_{x,v}=G_{x,s}$ holds.
Since $\Pic^{0}_{C(1_{x})/K(x^{1/p})}$ does not contain algebraic groups isomorphic to $\G_{a,K(x^{1/p})}$ again by \cite[Proposition 9.2.4]{BLR}, we have $G_{x,v}\subset G_{x}(1)$.
Next, we show that $G_{x}(1)\subset G_{x}$.
%Since any homomorphism from a connected smooth linear algebraic group to an abelian variety is trivial by \cite[Lemma 2.3]{C}, 
By Lemma \ref{fundamentalstructure}.4 and 5, the algebraic subgroup $G_{x}(1)\times_{\Spec K(x^{1/p})}\Spec\overline{K}$ is contained in $G$.
Moreover, by the definition of $G$ and $G_{x}$, we have $G\subset G_{x}\times_{\Spec K(x^{1/p})}\Spec\overline{K}$.

The desired equivalence follows from Lemma \ref{connsmunip} and Proposition \ref{pointsubgroup}.
\end{proof}

\label{simplecase}

\section{Examples}
\label{eg}
In this section, we give several examples of calculations of $\delta$-invariants.
In particular, we have the following.
\begin{prop}
The four cases appearing in Lemma \ref{qlem} actually occur.
Moreover, there exists an example where the first inequality of Lemma \ref{qlem}.3 (resp.\,Lemma \ref{qlem}.4) is strict. We also give an example where the first inequality of Lemma \ref{qlem}.3 (resp.\,Lemma \ref{qlem}.4) is an equality.
\end{prop}
In this section, we fix an algebraically closed field $k$ of characteristic $p$, and we put $K:= k(t)$.
Moreover, let $n \geq 2$ be an integer which is coprime to $p$.

\subsection*{(1): the case where $e_{1}=e_{2}=p$}
\label{ramram}
We put
\[
R :=
(K[X,Y]/ X^{n}-(Y^{p^{2}}-t))_{(X, Y^{p^{2}}-t)}.
\] 
To show that (this) $R$ satisfies the assumptions on $R$ in Section \ref{simplecase}, we use the following argument:\\
By the Jacobian criterion for varieties over $k$ (resp.\,for varieties over $K$), one can show that $R$ is regular and integral (resp.\,geometrically reduced over $K$).
Moreover, since the residue field of $R$ is purely inseparable over $K$, the field $\Frac R$ is geometrically connected over $K$.
Therefore, $R$ is geometrically integral over $K$.$\cdots(\ast)$\\
Then the normalization of $R\otimes_{K}K^{1/p}$ is 
\[
R(1) =
(K^{1/p}[Z,Y]/ Z^{n} - (Y^{p}-t^{1/p}))_{(Z, Y^{p}-t^{1/p})},
\]
where we put 
\[
Z: = \frac{(Y^{p}-t^{1/p})^{a}}{X^{b}}
\]
so that $Z^p =X$.
Here, $a,b$ are integers such that $na-pb =1$.
Therefore, we have $t^{1/p^{2}} \in L(1)$ and $e_{1}=e_{2}=p$.

\subsection*{(2): the case where $e_{1}=p$ and $e_{2} =1$}
We put 
\[
R :=
(K[X,Y]/ X^{p+n}-X^{p}-Y^{p}(Y^{p^{2}}-t))_{(X,Y^{p^{2}}-t)},
\]
which is regular and geometrically integral over $K$ by the argument $(\ast)$. 
The normalization of $R\otimes_{K}K^{1/p}$ is
\[
R (1) =
(K^{1/p}[Z,Y]/Z^{p+n}-Z^{p}-Y(Y^{p}-t^{1/p}))_{(Z, Y^{p}-t^{1/p})},
\]
where we put 
\[
Z:= \frac{(X+ Y(Y^{p}-t^{1/p}))^{a}}{X^{b}}
\]
so that $Z^p =X$.
Here, $a,b$ are integers such that $a(p+n) - bp =1$.
Therefore, we have $e_{1}=p$ and $e_{2} =1$.

\subsection*{(3): the case where $e_{1}= e_{2} =1$}
We put
\[
R :=
(K[X,Y]/X^{p^2}-Y(Y^{p}-t))_{(X, Y^{p}-t)},
\]
which is regular and geometrically integral over $K$ by the argument $(\ast)$.
The normalization of $R\otimes_{K}K^{1/p}$ is 
\[
R(1) =
(K^{1/p}[X,Z]/X^{p}-Z(Z^{p}-t^{1/p}))_{(X,Z^{p}-t^{1/p})},
\]
where we put
\[
Z:= \frac{X^{p}}{Y-t^{1/p}}
\]
so that $Z^p=Y$.
Therefore, we have $t^{1/p^{2}} \in L(1)$ and  $e_{1}= e_{2} =1$.
Moreover, we have $pq(t^{1/p})= q(t)$ (cf.\,Lemma \ref{qlem}.3).

We also give another example where the first inequality of Lemma \ref{qlem}.3 is strict (i.e.\,$L(2_{x})$ is not a simple extension field of $L$).
We put $K':= k(s,t)$.
We put
\[
R' :=
(K'[X,Y]/ YX^{p^{3}}-s^{p}X^{p^{2}}-(Y^{p}-t))_{(X, Y^{p}-t)},
\]
which is regular and geometrically integral over $K'$ by the argument $(\ast)$.
The normalization of $R\otimes_{K'}K'(t^{1/p})$ is
\[
R'(1_{t}) =
(K'(t^{1/p})[X,Z]/ ZX^{p^{2}}- sX^{p}-(Z^{p}-t^{1/p}))_{(X, Z^{p}-t^{1/p})},
\]
where we put
\[
Z := 
\frac{sX^{p}+ (Y-t^{1/p})}{X^{p^{2}}}
\]
so that $Z^p=Y$.
Therefore, we have $t^{1/p^{2}} \in L(1_{t})$, $q(t)= p^{3}$, and $q(t^{1/p})= p$.
Thus, we have $pq(t^{1/p})<q(t)$.

\subsection*{(4): the case where $e_{1}=1$ and $e_{2} =p$}

We put 
\[
R:= (K[X,Y]/X^{np}-Y(Y^{p}-t))_{(X,Y^{p}-t)},
\]
which is regular and geometrically integral over $K$ by the argument $(\ast)$.
The normalization of $R_{K^{1/p}}$ is 
\[
R(1)=(K^{1/p}[X,Z]/X^{n}-Z(Z^{p}-t^{1/p}))_{(X,Z^{p}-t^{1/p})},
\]
where we put 
\[
Z:=\frac{X^{n}}{Y-t^{1/p}}
\]
so that $Z^p=Y$.
In this case, we have $t^{1/p^2}\in L(1)$, $e_{1}=1$, and $e_{2}=p$.
Moreover, we have $q(t)=np$ and $q(t^{1/p})=n$.
Thus, we have $pq(t^{1/p})=q(t)$ (cf.\,Lemma \ref{qlem}.4).

We also give another example where the first inequality of Lemma \ref{qlem}.4 is strict.
We take a positive integer $m$ such that $m>n$.
We put
\[
R' :=
(K[X,Y]/ X^{mp} - YX^{np} - Y (Y^{p}-t))_{(X,Y^{p}-t)},
\]
which is regular and geometrically integral over $K$ by the argument $(\ast)$.
The normalization of $R'\otimes_{K}K^{1/p}$ is 
\[
R' (1) =
(K^{1/p}[X,Z]/ X^{m} - Z X^{n} - Z (Z^{p}-t^{1/p}))_{(X, Z^{p}-t^{1/p})},
\]
where we put
\[
Z:= \frac{X^{m}}{X^{n} + (Y-t^{1/p})}
\]
so that $Z^p=Y$.
Therefore, in this case we have $t^{1/p^{2}} \in L(1), e_{1}=1$, and $e_{2} =p$.
Moreover, we have $q (t) = mp$ and $q(t^{1/p}) = n$. 
%This gives an example where the equality in Lemma \ref{qlem} (4) does not hold.
Thus, we have $pq(t^{1/p})<q(t)$.

\section{Relation between genus changes and Jacobian numbers}
\label{jac}
Jacobian numbers, which have been studied by many researchers (for example, \cite{Buchweitz1980},
\cite{Esteves2003},
\cite{Greuel2007},
\cite{IIL},
and \cite{Tjurina1969}), are useful invariants of singularities of curves.
In this section, we give a comparison between Jacobian numbers and genus changes (or equivalently, $\delta$-invariants).

First, we recall the definition of continuous derivations to define the Jacobian number for a discrete valuation ring (of sufficiently general class) over a field.
See \cite[\S20.3 and \S20.7]{EGA41} for fundamental treatment of the notion of continuous derivations.
Let $K$ be a field, $R$ a Noetherian local ring over $K$, $\fm$ the maximal ideal of $R$, and $L$ the residue field of $R$.
In this section, in the case where $R$ is complete, for any finitely generated $R$-module $M$, we always consider $M$ as an $\fm$-adic topological $R$-module.

\begin{dfn}[cf.\,Remarks \ref{compareEGA}.1 and 2]
\label{defomega}
Suppose that $R$ is complete.
For any finitely generated $R$-module $M$, we write $\Der_{K}^{c}(R,M)$ for the set of continuous $K$-derivations from $R$ to $M$.
We define an $R$-module $\Omega^{1,c}_{R/K}$ to be a finitely generated $R$-module such that there exists an isomorphism of functors
\begin{equation*}
\begin{split}
\Hom_{R}(\Omega^{1,c}_{R/K},-)\simeq \Der_{K}^{c}(R,-)
\end{split}
\end{equation*}
from the category of finitely generated $R$-modules to the category of $R$-modules.
(Note that $\Omega^{1,c}_{R/K}$ does not always exist.)
\end{dfn}

\begin{rem}
\begin{enumerate}
\item
Regard $K$ as a discrete topological ring and $R$ as an $\fm$-adic topological ring.
In \cite[$0_{\text{IV}}$.20.7.14]{EGA41}, an $R$-module $\widehat{\Omega}^{1}_{R/K}$ is defined.
By the construction of $\widehat{\Omega}^{1}_{R/K}$, we have
$$\widehat{\Omega}^{1}_{R/K}\simeq \varprojlim_{n}\Omega^{1}_{R/K}\otimes_{R}(R/\fm^{n}).$$
By \cite[($0_{\text{IV}}$.20.7.14.4)]{EGA41}, if $\widehat{\Omega}^{1}_{R/K}$ is a finitely generated $\widehat{R}$-module, we have a canonical isomorphism
$$\widehat{\Omega}^{1}_{R/K}\simeq \Omega^{1,c}_{\widehat{R}/K}.$$
(Note that ``$\Omega^{1}_{R/K}$'' in the sense of \cite[D\'efinition $0_{\text{IV}}$.20.4.3]{EGA41} does not coincide with $\Omega^{1,c}_{R/K}$ in general.)
\item
Suppose that $R$ is complete and $L$ is finitely generated over $K$.
The existence of the module $\Omega^{1,c}_{R/K}$ follows from \cite[Proposition $0_{\text{IV}}$.20.7.15]{EGA41} and Remark \ref{compareEGA}.1.
Here, we give an explicit construction of $\Omega^{1,c}_{R/K}$.
Suppose that we can take a surjective $K$-homomorphism
\begin{align}
\varphi\colon K[[T_{1},\ldots,T_{m}]][X_{1},\ldots,X_{n}]\to R
\label{presentation}
\end{align}
such that the image of $T_{i}$ is contained in $\fm$ for each $i$.
We choose a system of generators $f_{1},\ldots,f_{l}$ of the kernel of $\varphi$.
Then the module
\begin{equation}
\label{equationOmega}
M:=
\frac{\bigoplus_{1\leq i\leq m}RdT_{i}
\oplus\bigoplus_{1\leq j\leq n}RdX_{j}}{\langle \sum_{1\leq i\leq m}\frac{\partial f_{t}}{\partial T_{i}}dT_{i}
+\sum_{1\leq j\leq n}\frac{\partial f_{t}}{\partial X_{j}}dX_{j}\mid 1\leq t\leq l\rangle}
\end{equation}
and a group homomorphism $d\colon R\to M$ sending $f$ to 
\begin{equation}
\label{equationd}
\sum_{1\leq i\leq m}\frac{\partial f}{\partial T_{i}}dT_{i}
+\sum_{1\leq j\leq n}\frac{\partial f}{\partial X_{j}}dX_{j}
\end{equation}
satisfy the condition of the definition of $\Omega^{1,c}_{R/K}$.
\item
If $R$ is essentially of finite type over $K$, we have a canonical isomorphism
$$\Omega^{1}_{R/K}\otimes_{R}\widehat{R}\simeq \Omega^{1,c}_{\widehat{R}/K}$$
by Remark \ref{compareEGA}.1.
\item
As in Remark \ref{compareEGA}.2, suppose that $R$ is complete and has the presentation (\ref{presentation}).
Moreover, suppose that $R$ is a domain.
Then $\Frac(R)$ (resp.\,any finite dimensional $\Frac(R)$-linear space $V$) has a canonical topological field structure (resp.\,a canonical topological $\Frac(R)$-linear space structure) such that the topology of $R$ (resp.\,any finitely generated $R$-submodule of $V$) coincides with the relative topology from $\Frac(R)$ (resp.\,$V$).
Let $D\colon \Frac(R)\to V$ be a continuous derivation over $K$ to a finite dimensional $\Frac(R)$-linear space.
Then $D(R)$ is contained in the $R$-subsmodule generated by all $D(T_{i})$ and $D(X_{j})$.
Hence, we have a natural isomorphism of functors
$$\Hom_{\Frac(R)}(\Omega^{1,c}_{R/K}\otimes_{R}\Frac(R),-)\simeq \Der_{K}^{c}(\Frac(R),-)$$
from the category of finite dimensional $\Frac(R)$-linear spaces to the category of $\Frac(R)$-linear spaces.
Here, $\Der_{K}^{c}(\Frac(R),-)$ is the set of continuous $K$-derivations.

Let $R'$ be an integral extension ring $R'\supset R$ contained in $\Frac(R)$.
Note that $R'$ is complete local and finite over $R$ since $R$ is complete.
Then we have a natural isomorphism $\Omega^{1,c}_{R/K}\otimes_{R}\Frac(R)\simeq \Omega^{1,c}_{R'/K}\otimes_{R'}\Frac(R)$.
\item
For later use, we use a similar convention for the product of complete local rings in the following way:
Let $R'$ be a Noetherian ring over $K$.
Suppose that $R'$ is the product of finitely many complete local rings $R_{k}$.
For any finitely generated $R'$-module $M$, we have a canonical decomposition $M = \prod_{k} M_{k}$, where $M_{k}$ is a finitely generated $R'$-module supported in $\Spec R_{k}$.
We put $\Der_{K}^{c} (R', M) := \prod_{k} \Der_{K}^{c} (R_{k}, M_{k})$.
As in Definition \ref{defomega}, if $\Der_{K}^{c} (R', -)$ is represented by a finitely generated $R'$-module, we denote it by $\Omega_{R'/K}^{1,c}$.
Then $\Omega_{R'/K}^{1,c}$ exists if and only if, for all $k$, the modules $\Omega_{R_{k}/K}^{1,c}$ exist.
Moreover, we have
\[
\Omega_{R'/K}^{1,c} \simeq \prod_{k} \Omega_{R_{k}/K}^{1,c}
\]
if $\Omega_{R'/K}^{1,c}$ exists.
As in Remark \ref{compareEGA}.2, if we have a surjective $K$-homomorphism
\begin{align}
\varphi\colon K[[T_{1},\ldots,T_{m}]][X_{1},\ldots,X_{n}]\to R'
\end{align}
such that the $k$-component of the image of $T_{i}$ is contained in the maximal ideal of $R_{k}$ for any $i,k$,
then the $R'$-module $M$ and $d\colon R' \rightarrow M$ that are defined by the formulas (\ref{equationOmega}) and (\ref{equationd}) satisfy the condition in the definition of $\Omega_{R'/K}^{1,c}$.
In this case, for any finite field extension $K'/K$, we have
$R_{k}\otimes_{K} K' \simeq \prod_{k'}S_{k,k'}$,
where $S_{k'}$ are complete local rings over $K'$.
Therefore, $R' \otimes_{K} K'$ is also the product of finite complete local rings.
By the above construction, we have
\[
\Omega^{1,c}_{R'/K}\otimes_{K}K' \simeq \Omega^{1,c}_{R' \otimes_{K} K'/K'}.
\]
\end{enumerate}
\label{compareEGA}
\end{rem}

\begin{dfn}
Suppose that $R$ is $1$-dimensional and $L$ is finite over $K$.
We define {\it the Jacobian number} of $R$ over $K$ to be
$$\jac (R):=\dim_{K}(\widehat{R}/\mathrm{Fitt}_{1}\Omega^{1,c}_{\widehat{R}/K}) \in \Z_{\geq 0} \cup \{\infty\}.$$
Here, ``$\mathrm{Fitt}_{1}-$" denotes the first Fitting ideal of the $R$-module.
\end{dfn}

\begin{rem}
Suppose that $R$ is essentially of finite type over $K$.
Then $\jac(R)$ agrees with the Jacobian number defined in \cite[Definition 4.1]{IIL}
by Remark \ref{compareEGA}.3.
\end{rem}

In the rest of this section, we consider the following situation:
Let $K$ be a field of characteristic $p$ satisfying $[K^{1/p} : K] < \infty$, $R$ a discrete valuation ring over $K$ which is geometrically integral over $K$, $\varpi$ a uniformizer of $R$, and $L$ the residue field of $R$.
We suppose that $L/K$ is a finite extension.
Let $R(1)$ be the normalization of $R\otimes_{K}K^{1/p}$.
We further suppose that $R(1)$ is finite over $R\otimes_{K}K^{1/p}$.
We write $g_{10}$ for the natural number
\[
\dim_{K^{1/p}}(R(1)/(R\otimes_{K}K^{1/p})).
\]
We note that we use the same notation as in Section \ref{geomnormalsection}.
The main theorem in this section is the following:

\begin{thm}
\label{genjacob}
In the above situation, we have
\[
\frac{g_{10}}{(p-1)/2} = \frac{\jac (R)}{p}.
\]
\end{thm}

The rest of this section is devoted to proving Theorem \ref{genjacob}.

First, we may assume that $R$ is complete by the assumption and Proposition \ref{B-prop}.7.
Since 
\[
(\mathrm{Fitt}_{1} \Omega_{R/K}^{1,c})\otimes_{K}K' \simeq \mathrm{Fitt}_{1} (\Omega_{R/K}^{1,c} \otimes_{K} K')
\]
for any finite field extension $K'$ over $K$ by Remark \ref{compareEGA}.5, we may assume $L$ is purely inseparable over $K$.
Moreover, we have
\[
\Omega^{1,c}_{(R\otimes_{K}K^{\sep})^{\wedge}/K^{\sep}} \simeq \Omega^{1,c}_{R/K}\otimes K^{\sep}
\]
by the construction in Remark \ref{compareEGA}.2. (Here, $(R\otimes_{K}K(x^{1/p}))^{\wedge}$ denotes the completion of $R\otimes_{K}K(x^{1/p})$.)
Therefore, by taking the base change and the completion again, we may assume that $K$ is separably closed and $R$ is complete.

We write $\Omega^{1,c}_{R/K,\mathrm{tor}}$ for the torsion part of $\Omega^{1,c}_{R/K}$.

\begin{lem}
\label{lemtorsion}
Write $\phi$ for the natural homomorphism 
\[
\Omega^{1,c}_{R/K} \otimes_{R} R(1) \rightarrow \Omega^{1,c}_{R(1)/K^{1/p}}.
\]
Then we have
$$\langle dr\mid r\in R(1)^{p}\rangle_{R}=\Omega^{1,c}_{R/K,\mathrm{tor}}
\quad\text{and}\quad
\Omega^{1,c}_{R/K,\mathrm{tor}} \otimes_{R} R(1) =  \Ker \phi,$$
and the induced homomorphism
\begin{align}
\Omega^{1,c}_{R/K}\otimes_{R}\Frac(R(1))\to\Omega^{1,c}_{R(1)/K^{1/p}}\otimes_{R(1)}\Frac(R(1))
\label{aftertensor}
\end{align}
is an isomorphism of $1$-dimensional $\Frac(R(1))$-linear spaces.
\end{lem}

\begin{proof}
The homomorphism (\ref{aftertensor}) is an isomorphism by Remark \ref{compareEGA}.2 and Remark \ref{compareEGA}.4.
From this and the calculation $dr^{p}=pr^{p-1}dr=0$, we have
$$\langle dr\mid r\in R(1)^{p}\rangle_{R} \otimes_{R}R(1) \subset \Ker \phi\subset\Omega^{1,c}_{R/K, \mathrm{tor}} \otimes_{R}R(1).$$
Since $R$ is complete, we have $[R(1):R(1)^{p}] = [K^{1/p}:K]p$.
Therefore, there exists an element $t \in R(1)^{p}$ such that 
\[
R = R(1)^{p}[t^{1/p}] \simeq R(1)^{p}[X]/(X^{p}-t).
\]
By Remark \ref{compareEGA}.2, we have
\begin{align}
\Omega^{1,c}_{R/K}
&\simeq((\Omega^{1,c}_{R(1)^{p}/K}\otimes_{R(1)^{p}}R)\oplus RdX)/Rd(X^{p}-t)\\
&\simeq(\Omega^{1,c}_{R(1)^{p}/K}\otimes_{R(1)^{p}}R/Rdt)\oplus Rdt^{1/p}.
\label{omegacomputation}
\end{align}
In particular, $dt^{1/p}$ is torsion free in $\Omega^{1,c}_{R/K}$.
This shows that linear spaces in (\ref{aftertensor}) are $1$-dimensional and
\[
\Omega^{1,c}_{R/K, \mathrm{tor}}\subset  (\Omega^{1,c}_{R(1)^{p}/K} \otimes_{R(1)^{p}}R)/Rdt=\langle dr\mid r\in R(1)^{p}\rangle_{R}
\]
holds.
We finish the proof of Lemma \ref{lemtorsion}.
\end{proof}

\begin{lem}[cf.\,\cite{Rim} and {\cite[Proposition 4.7]{IIL}}]
\label{lemjactor}
The equality
\[
\jac (R) = \dim_{K} (\Omega^{1,c}_{R/K,\mathrm{tor}})
\]
holds.
\end{lem}

\begin{proof}
Since we have $\dim_{\Frac(R)}(\Omega^{1,c}_{R/K}\otimes_{R}\Frac(R))=1$ by Lemma \ref{lemtorsion}, this lemma follows from the structure theorem for finitely generated modules over a principal ideal domain.
\end{proof}

We take an absolute $p$-basis $x_{1}, \ldots, x_{c}$ of $K$.
For $0 \leq i \leq c$, let $K_{i}$ be the field
$K(x_{1}^{1/p}, \ldots, x_{i}^{1/p})$
and $R_{i}$ the normalization of $R\otimes_{K} K_{i}$.
Let $L_{i}$ be the residue field of $R_{i}$.
We note that $K_{0}=K, R_{0} =R,$ and $L_{0} = L$.

\begin{lem}
\label{relemkernel}
\begin{enumerate}
\item
There exists an element $a\in R_{1}$ satisfying $R_{1}=R[a]$.
\item
Let $a$ be such an element of $R_{1}$ and $f$ the natural homomorphism 
\[
\Omega^{1,c}_{R/K} \otimes_{R} R_{1} \rightarrow \Omega^{1,c}_{R_{1}/K_{1}}.
\]
Then we have
\[
\Ker f = R_{1}da^{p}.
\]
\item
Suppose $x_{1}^{1/p} \in L$.
Then we can define $q(x_{1})$ in Notation \ref{notationq}.
If $p \mid q(x_{1})$, then there exists an isomorphism
\[
\Ker f \simeq R_{1}/ (\varpi^{q(x_{1})}).
\]
If $p \nmid q(x_{1})$, then there exists an isomorphism
\[
\Ker f \simeq R_{1}/ (\varpi^{q(x_{1})-1}).
\]
\end{enumerate}
\end{lem}

\begin{proof}
We can choose $a$ to be a lift of a generator of the residue field of $R_{1}$ over $L$ (resp.\,a uniformizer of $R_{1}$) by Lemma \ref{choiceofr} (resp.\,\cite[Proposition 17]{Ser}) in the case where $p \mid q(x_{1})$ (resp.\,$p \nmid q(x_{1})$).
Then we have
\[
\Ker (\Omega_{R/K}^{1,c}\otimes_{R}R_{1} \rightarrow \Omega_{R_{1}/K}^{1,c}) = R_{1}da^{p}
\]
by Remark \ref{compareEGA}.2 (cf.\, the calculation (\ref{omegacomputation})).
Moreover, by Remark \ref{compareEGA}.2, we have an exact sequence
\[
\Omega_{K_{1}/K}^{1,c} \otimes_{K_{1}} R_{1} \rightarrow  \Omega_{R_{1}/K}^{1,c} \rightarrow \Omega_{R_{1}/K_{1}}^{1,c}
\]
and $\Omega_{K_{1}/K}^{1,c}$ is free of rank $1$ over $K_{1}$.
Therefore, the homomorphism
\[
\Ima (\Omega_{R/K}^{1,c}\otimes_{R}R_{1} \rightarrow \Omega_{R_{1}/K}^{1,c}) \rightarrow \Omega_{R_{1}/K_{1}}^{1,c} 
\]
is injective since $\Omega_{R/K}^{1,c} \otimes_{R} \Frac (R_{1}) \rightarrow \Omega^{1,c}_{R_{1}/K_{1}} \otimes_{R_{1}} \Frac (R_{1})$ is an isomorphism by Remarks \ref{compareEGA}.4 and \ref{compareEGA}.5.
Now we have 
\[
\Ker f = R_{1}da^{p}.
\]

Assertion 3 follows from Lemma \ref{lemkernel}, which we will prove at the end of this section.
\end{proof}

\begin{lem}
\label{lemgenmain}
We fix an integer $i$ with $1 \leq i \leq c$.
Let
\[
\xymatrix
{
\Omega^{1,c}_{R/K} \otimes_{R} R_{i} \ar[r]^-{f} 
&\Omega^{1,c}_{R_{i-1}/K_{i-1}} \otimes_{R_{i-1}} R_{i} \ar[r]^-{g} 
&\Omega^{1,c}_{R_{i}/K_{i}}
}
\]
be the natural homomorphisms.
Then the sequence
\[
\xymatrix
{
0 \rightarrow \Ker f  \ar[r]  
&\Ker g\circ f \ar[r]^-{f} 
&\Ker g \rightarrow 0
}
\]
is exact.
\end{lem}

\begin{proof}
We only need to show the exactness at $\Ker g$.
Take an element $a\in R_{i}$ satisfying $R_{i}=R_{i-1}[a]$ by Lemma \ref{relemkernel}.1.
By Lemma \ref{relemkernel}.2, $\Ker g$ is generated by $da^{p}$.
Since we have $a^{p}\in R$, Lemma \ref{lemgenmain} holds.
\end{proof}

\begin{proof}[Proof of Theorem \ref{genjacob} (assuming Lemma \ref{relemkernel}.3)]
Now we start the proof of Theorem \ref{genjacob} assuming Lemma \ref{relemkernel}.3.
We have
\begin{eqnarray*}
\jac (R) &=& \dim_{K}(\Omega_{R/K,\mathrm{tor}}^{1,c})\\
&=& \dim_{K^{1/p}} \Ker \phi\\
&=& \sum_{i=0}^{c-1} \dim_{K_{i}} \Ker (\phi_{i}),\\
\end{eqnarray*}
where $\phi_{i}$ is the natural homomorphism
\[
\Omega^{1,c}_{R_{i-1}/K_{i-1}} \otimes_{R_{i-1}} R_{i} \rightarrow \Omega^{1,c}_{R_{i}/K_{i}}.
\]
Here, the first equality follows from Lemma \ref{lemjactor}, the second equality follows from Lemma \ref{lemtorsion}, and the third equality follows from Lemma \ref{lemgenmain}.
By Lemma \ref{calcul}.1 and Lemma \ref{relemkernel}.3, we have
\[
\dim_{K_{i}} \Ker (\phi_{i}) = 
\begin{cases}
0 &(x_{i}^{1/p} \notin L_{i-1}),\\
\dim_{K_{i}} R_{i}/ \varpi_{R_{i-1}}^{q^{(i)}_{1}} = [L_{i-1}:K_{i}] p q^{(i)}_{1} &(p\mid q^{(i)}_{1}),\\
\dim_{K_{i}} R_{i}/ \varpi_{R_{i-1}}^{q^{(i)}_{1}-1} = [L_{i-1}:K_{i}] p (q^{(i)}_{1}-1) &(p \nmid q^{(i)}_{1}),
\end{cases}
\]
where we write $q^{(i)}_{1}$ for the invariant $q(x_{i})$ for $R_{i-1}$ (cf.\,Notation \ref{notationq}) and $\varpi_{R_{i-1}}$ for a uniformizer in $R_{i-1}$.
On the other hand, by Theorem \ref{q}, the genus changes 
\[
g_{i} :=\dim_{K_{i}} R_{i}/(R_{i-1}\otimes_{K_{i-1}} K_{i})\quad (1\leq i\leq c)
\] 
satisfy
\[
g_{i}
=
\begin{cases}
0 &(x_{i}^{1/p} \notin L_{i-1}),\\
[L_{i-1}:K_{i}] \frac{p-1}{2}q_{1}^{(i)} &(p\mid q^{(i)}_{1}),\\
[L_{i-1}:K_{i}] \frac{p-1}{2}(q_{1}^{(i)}-1)  &(p \nmid q^{(i)}_{1}).
\end{cases}
\]
Since $g_{10} = \sum_{i=0}^{c-1} g_{i}$, we have the desired equality. 
It finishes the proof of Theorem \ref{genjacob} up to the proof of Lemma \ref{relemkernel}.3.
\end{proof}

To complete the proof of Theorem \ref{genjacob}, we prove Lemma \ref{relemkernel}.3.
To understand the structure of $\Omega_{R/K}^{1,c}$, we need to describe the structure of $R$ in terms of invariants that are similar to $q(x)$.

\begin{nota}
\label{rbqb}
\begin{enumerate}
\item 
Let $y_{1}, \ldots, y_{m}$ be a $p$-basis of $L$ over $K$.
We put 
\[
n_{i} :=  \min{\{n \in \Z_{> 0} \,|\,y_{i}^{p^{n}}\in K(y_{1}, \ldots, y_{i-1})\}}.
\]
We put $z_{1} := y_{1}^{p^{n_{1}}}$. 
For $1 \leq i \leq m$, we fix $f_{i} \in K[T_{1}, \ldots, T_{i-1}]$ such that 
\[
f_{i}(y_{1}, \ldots, y_{i-1}) =y_{i}^{p^{n_{i}}}
\]
and $f_{1}= z_{1}$.
\item
For $1 \leq i \leq m$, we define elements $r'_{i}\in R$ and natural numbers $q_{i}$ and $q'_{i}$ inductively.
First, we put $q_{1} := \underset{r}{\max}\,v(r^{p}-z_{1})$, where $r\in R$ ranges over all the lifts of $y_{1}^{p^{n_{1}-1}}$.
Moreover, we put $q'_{1} := \underset{r}{\max}\,v(r^{p^{n_{1}}}-z_{1})$, where $r\in R$ ranges over all the lifts of $y_{1}$.
We fix $r'_{1} \in R$ such that $v(r_{1}^{\prime p^{n_{1}}}-z_{1}) = q'_{1}$.

We suppose that $q_{j}, q'_{j},$ and $r'_{j}$ are defined for $j= 1, \ldots, i-1$.
We put 
\[
q_{i} := \underset{r}{\max}\,v(r^{p}- f_{i} (r'_{1}, \ldots, r'_{i-1})),
\] 
where $r\in R$ ranges over all the lifts of $y_{i}^{p^{n_{i}}-1}$.
We also put 
\[
q'_{i} := \underset{r}{\max}\,v(r^{p^{n_{i}}}- f_{i} (r'_{1}, \ldots, r'_{i-1})), 
\]
where  $r\in R$ ranges over all the lifts of $y_{i}$.
We also fix $r'_{i} \in R$ such that 
\[v(r_{i}^{\prime p^{n_{i}}}-f_{i} (r'_{1}, \ldots, r'_{i-1})) = q'_{i}.
\]
Now $q_{i}, q_{i}'$, and $r_{i}'$ are defined.
\item
For $1 \leq i \leq m$, we take $u_{i}' \in R^{\times}$ such that
\[
r_{i}^{\prime p^{n_{i}}}- f_{i} (r'_{1}, \ldots, r'_{i-1}) = u_{i}' \varpi^{q'_{i}}.
\]
We also take $r_{i}\in R$ and a unit $u_{i} \in R^{\times}$ such that
\begin{equation}
\label{r_{i}}
r_{i}^{p}-f_{i} (r'_{1}, \ldots, r'_{i-1}) = u_{i} \varpi^{q_{i}}.
\end{equation}
We note that we can take $r_{i} = r_{i}'^{p^{n_{i}}-1}$ if $q_{i} = q_{i}'$ holds.
Therefore, we always assume this condition in the following.
\end{enumerate}
\end{nota}

\begin{thm}
\label{lemgen}
There exists a $K$-algebra isomorphism
\[
K[[S]][T_{1}, \ldots, T_{m}]/(T_{i}^{p^{n_{i}}} - f_{i} (T_{1}, \ldots, T_{i-1}) - \widetilde{u}_{i} S^{q_{i}} + \widetilde{w}_{i}^{p} S^{q'_{i}})_{1 \leq i \leq m}
\simeq R.
\]
Here, $\widetilde{u}_{i}$ and $\widetilde{w}_{i}$ are elements of $K[[S]][T_{1}, \ldots, T_{m}]$, and $\widetilde{u}_{i}$ (resp.\,$S$) goes to the unit $u_{i} \in R^{\times}$ which we took in (\ref{r_{i}}) (resp.\,the element $\varpi\in R$).
Moreover, the following hold.
\begin{itemize}
\item For each $1 \leq i \leq m$, if $q_{i} = q'_{i}$, then we can take $\widetilde{w}_{i}$ to be $0$.
\item If $p\nmid q_{1}$, by replacing $\varpi$ by another uniformizer, we can take $\widetilde{u}_{1}$ to be $1$.
\end{itemize}
In the following, we denote the above polynomial
\[
T_{i}^{p^{n_{i}}} - f_{i} (T_{1}, \ldots, T_{i-1}) - \widetilde{u}_{i} S^{q_{i}} + \widetilde{w}_{i}^{p} S^{q'_{i}}
\]
by $P_{i}$.
\end{thm}
\begin{proof} 
Let 
\[
\varphi\colon K[[S]][T_{1}, \ldots, T_{n}] \rightarrow R
\]
be the $K$-algebra homomorphism which sends $S$ to $\varpi$
and $T_{i}$ to $r'_{i}$.
Since the image of $\varphi$ contains a uniformizer and maps onto the residue field $L$ (note that $L = \bigoplus_{0 \leq i_{j} < p^{n_{j}}} K y_{1}^{i_{1}} \cdots y_{m}^{i_{m}}$), the homomorphism $\varphi$ is surjective.
By the definition of $u_{i}'$ and $u_{i}$, we have
\begin{equation}
\frac{r_{i}^{p}-r_{i}^{\prime p^{n_{i}}}}{\varpi^{q'_{i}}} = u_{i} \varpi^{q_{i}-q'_{i}} - u_{i}'.
\label{compareqq'}
\end{equation}
If $p \nmid q'_{i}$, then we have $q_{i} = q'_{i}$ (by the same argument as that in Lemma \ref{q<q}), and the left-hand side of (\ref{compareqq'}) is $0$ by the choice of $r_{i}$ (cf.\,Notation \ref{rbqb}.3).
Therefore, we can always take the $p$-th root $w_{i}$ of the left-hand side of (\ref{compareqq'}).
Then we have 
\[
r_{i}^{\prime p^{n_{i}}} - f_{i} (r'_{1}, \ldots, r'_{i-1})  = u_{i}' \varpi^{q'_{i}} = u_{i} \varpi^{q_{i}} - w_{i}^{p} \varpi^{q'_{i}}.
\]
Take elements $\widetilde{u}_{i} \in \varphi^{-1} (u_{i})$ and $\widetilde{w}_{i} \in \varphi^{-1} (w_{i})$.
Then $\varphi$ induces a surjective homomorphism
\[
K[[S]][T_{1}, \ldots, T_{m}]/(T_{i}^{p^{n_{i}}} -  f_{i} (T_{1}, \ldots, T_{i-1}) - \widetilde{u}_{i} S^{q_{i}} + \widetilde{w}_{i}^{p} S^{q'_{i}})_{1 \leq i \leq m}
\rightarrow R,
\]
which is an isomorphism since both sides are free modules over $K[[S]]$ of the same rank.
We note that if $p \nmid q_{1}$, we have $u_{1}=1$ after replacing $\varpi$ by a suitable uniformizer, since $L$ is separably closed.
In this case, we can take $\widetilde{u}_{1}$ to be $1$.
\end{proof}

We note that, for any $x_{1} \in (K\cap L^{p}) \setminus K^{p}$, there exists a $p$-basis $y_{1}, \ldots, y_{m}$ of $L$ over $K$ such that $z_{1} =x_{1}$.
Therefore, the following lemma gives the proof of Lemma \ref{relemkernel}.3.

\begin{lem}[cf.\,Lemma \ref{relemkernel}.3]
\label{lemkernel}
We denote $K(z_{1}^{1/p})$ by $K'$ and the normalization of $R\otimes_{K}K'$ by $R'$.
Let $f$ be the natural homomorphism 
\[
\Omega^{1,c}_{R/K} \otimes_{R} R' \rightarrow \Omega^{1,c}_{R'/K'}.
\]
If $p \mid q_{1}$, then we have
\[
\Ker f \simeq R' d u_{1} \simeq R'/(\varpi^{q_{1}}).
\]
On the other hand, if $p \nmid q_{1}$ and 
$\widetilde{u}_{1} = 1$, then we have
\[
\Ker f \simeq R' d \varpi \simeq R'/ (\varpi^{q_{1}-1}).
\]
\end{lem}

\begin{proof}
In the proof, we use the notation in Theorem \ref{lemgen}.
By the proof of Lemma \ref{relemkernel}.1, we have 
$
R' = R[u_{1}^{1/p}]
$ (resp.\,$R' = R[\varpi^{1/p}]$) 
in the case where $p \mid q_{1}$ (resp.\,$p\nmid q_{1}$).
Moreover, by Lemma \ref{relemkernel}.2, we have
\[
\Ker f \simeq R' du_{1}\quad(\text{resp.}\,\Ker f \simeq R'd\varpi).
\]
Therefore, it suffices to show that $ Rdu_{1} \simeq R/ (\varpi^{q_{1}})$ (resp.\,$Rd\varpi \simeq R/ \varpi^{q_{1}-1}$) in $\Omega^{1,c}_{R/K}$.
Since $p \mid q_{1}$ (resp.$\,p\nmid q_{1}$) and the image of
\[
dP_{1}= - S^{q_{1}}d \widetilde{u}_{1}\quad (\text{resp.}\,dP_{1}=-q_{1}S^{q_{1}-1}dS)
\]
is $0$ in $\Omega^{1,c}_{R/K}$, we have 
\[
\varpi^{q_{1}}d u_{1}=0\quad(\text{resp.}\,\varpi^{q_{1}-1}d\varpi=0).
\]
Here, we note that if $p\nmid q'_{1}$ holds, then we have $q'_{1}=q_{1}$ and $\widetilde{w}_{1}=0$.
By Remark \ref{compareEGA}.2 and Theorem \ref{lemgen}, we have 
\[
\Omega_{R/K}^{1,c} \simeq \frac{R d S \oplus \bigoplus_{1 \leq i \leq m} R d T_{i}}{\langle dP_{i} \rangle_{1 \leq i \leq m}}.
\]
By Lemma \ref{lemtorsion}, $\Omega_{R/K}^{1,c} \otimes_{R}\Frac (R)$ is of rank $1$.
Therefore, $dP_{1}$ is not contained in the submodule generated by $d P_{j} \in R d S \oplus \bigoplus_{1 \leq i \leq m} R d T_{i}$ $(2 \leq j \leq m)$, and it finishes the proof.
\end{proof}

\newcommand{\etalchar}[1]{$^{#1}$}
\providecommand{\bysame}{\leavevmode\hbox to3em{\hrulefill}\thinspace}
\providecommand{\MR}{\relax\ifhmode\unskip\space\fi MR }
% \MRhref is called by the amsart/book/proc definition of \MR.
\providecommand{\MRhref}[2]{%
  \href{http://www.ams.org/mathscinet-getitem?mr=#1}{#2}
}
\providecommand{\href}[2]{#2}


\begin{thebibliography}{Stacks99}
\bibitem[A]{A} R.~Achet, \emph{The {P}icard group of the forms of the affine line and of the
  additive group}, J. Pure Appl. Algebra \textbf{221} (2017), no.~11,
  2838--2860.
\bibitem[BG]{Buchweitz1980}
R.-O. Buchweitz and G.-M. Greuel, \emph{The {M}ilnor number and deformations of
  complex curve singularities}, Invent. Math. \textbf{58} (1980), no.~3,
  241--281. 
\bibitem[BLR]{BLR} S.~Bosch, W.~L\"{u}tkebohmert, and M.~Raynaud, \emph{N\'{e}ron models},
  Ergebnisse der Mathematik und ihrer Grenzgebiete (3) [Results in Mathematics
  and Related Areas (3)], vol.~21, Springer-Verlag, Berlin, 1990.
\bibitem[CGP]{CGP} B.~Conrad, O.~Gabber, and G.~Prasad, \emph{Pseudo-reductive groups}, second
  ed., New Mathematical Monographs, vol.~26, Cambridge University Press,
  Cambridge, 2015.
\bibitem[EGA41]{EGA41}
A.~Grothendieck, 
  {\emph{\'{E}l\'{e}ments de g\'{e}om\'{e}trie alg\'{e}brique. {IV}. \'{E}tude
  locale des sch\'{e}mas et des morphismes de sch\'{e}mas. {I}}}, Inst. Hautes
  \'{E}tudes Sci. Publ. Math. (1964), no.~20, 259 pp.
\bibitem[EGA42]{EGA42} 
A.~Grothendieck, 
  {\emph{\'{E}l\'{e}ments de g\'{e}om\'{e}trie alg\'{e}brique. {IV}. \'{E}tude
  locale des sch\'{e}mas et des morphismes de sch\'{e}mas. {II}}}, Inst. Hautes
  \'{E}tudes Sci. Publ. Math. (1965), no.~24, 231 pp. 
\bibitem[EK]{Esteves2003}
E.~Esteves and S.~Kleiman, \emph{Bounds on leaves of one-dimensional
  foliations}, vol.~34, 2003, Dedicated to the 50th anniversary of IMPA,
  pp.~145--169.
\bibitem[FGA]{FGA} A. Grothendieck, {\it Fondements de la G\'eom\'etrie Alg\'ebrique} Sem. Bourbaki, exp. $\mathrm{n}^{\circ}$ 149 (1956/57), 182 (1958/59), 190 (1959/60), 195 (1959/60), 212 (1960/61), 221 (1960/61), 232 (1961/62), 236 (1961/62), Benjamin, New York (1966)
\bibitem[GLS]{Greuel2007}
G.-M. Greuel, C.~Lossen, and E.~Shustin, \emph{Introduction to singularities
  and deformations}, Springer Monographs in Mathematics, Springer, Berlin,
  2007.
\bibitem[IIL]{IIL}
K.~Ito, T.~Ito, and C.~Liedtke, \emph{Deformations of rational curves in
  positive characteristic}, J. Reine Angew. Math. \textbf{769} (2020), 55--86.
\bibitem[Mil]{Milne2017}
J.~S. Milne, \emph{Algebraic groups: The
  theory of group schemes of finite type over a field}, 
Cambridge Studies in Advanced
  Mathematics, vol. 170, Cambridge University Press, Cambridge, (2017).
  
\bibitem[PW]{PW} Z.~{Patakfalvi} and J.~{Waldron}, \emph{{Singularities of General Fibers and
  the LMMP}}, preprint (2017), arXiv:1708.04268, to appear in Am. J. Math.
   \bibitem[Rim]{Rim} D.~S. Rim, \emph{Torsion differentials and deformation}, Trans. Amer. Math.
  Soc. \textbf{169} (1972), 257--278.
\bibitem[Sc]{Sc} S.~Schr\"{o}er, \emph{On genus change in algebraic curves over imperfect
  fields}, Proc. Amer. Math. Soc. \textbf{137} (2009), no.~4, 1239--1243.
  
\bibitem[Ser1]{Ser1} J.-P. Serre, \emph{Groupes alg\'{e}briques et corps de classes}, Publications
  de l'Institut de Math\'{e}matique de l'Universit\'{e} de Nancago, VII,
  Hermann, Paris, 1959.
\bibitem[Ser2]{Ser} \bysame, \emph{Local fields}, Graduate Texts in Mathematics, vol.~67,
  Springer-Verlag, New York-Berlin, 1979, Translated from the French by Marvin
  Jay Greenberg.
  \bibitem[SGA3]{SGA3} M. Demazure, A. Grothendieck, {\it Sch\'emas en groupes I, II, III}, Lecture Notes
in Math \textbf{151}, \textbf{152}, \textbf{153}, Springer-Verlag, New York (1970).
  \bibitem[Sp]{Sp} T.~A. Springer, \emph{Linear algebraic groups}, second ed., Progress in
  Mathematics, vol.~9, Birkh\"{a}user Boston, Inc., Boston, MA, 1998.
Birka\"user.
\bibitem[Stacks]{Stacks} The Stacks project, \texttt{https://stacks.math.columbia.edu}.
\bibitem[Tan]{Tanaka2019}
H.~{Tanaka}, \emph{{Invariants of algebraic varieties over imperfect fields}},
  preprint (2019), arXiv:1903.10113.
\bibitem[Tat]{Ta} J.~Tate, \emph{Genus change in inseparable extensions of function fields},
  Proc. Amer. Math. Soc. \textbf{3} (1952), 400--406.
\bibitem[Tju]{Tjurina1969}
G.~N. Tjurina, \emph{Locally semi-universal flat deformations of isolated
  singularities of complex spaces}, Izv. Akad. Nauk SSSR Ser. Mat. \textbf{33}
  (1969), 1026--1058.
\end{thebibliography}
\end{document}